\documentclass[reqno, english,11pt]{amsart}

\usepackage{mathtools}
\mathtoolsset{showonlyrefs}

\DeclareSymbolFont{bbold}{U}{bbold}{m}{n}
\DeclareSymbolFontAlphabet{\mathbbold}{bbold}

\usepackage{color}
\usepackage[colorinlistoftodos]{todonotes}

\usepackage{pifont}

\RequirePackage{mathrsfs} \let\mathcal\mathscr
\usepackage{a4wide}
\usepackage[all]{xy}
\usepackage{enumerate}
\usepackage{color}
\usepackage{amssymb,amsmath,amsfonts,amsthm}

\usepackage{hyperref}
\usepackage{dsfont}
\usepackage{upgreek}
\usepackage{mathrsfs}
\usepackage{mathabx,yfonts}
\usepackage{enumerate}
\usepackage{enumitem}
\usepackage{graphicx}

\newtheorem{theorem}{Theorem}

\newtheorem{lemma}[theorem]{Lemma}
\newtheorem{proposition}[theorem]{Proposition}

\theoremstyle{definition}

\newtheorem*{notation}{Notation}

\numberwithin{theorem}{section}
\numberwithin{equation}{section}
\numberwithin{table}{section}

\DeclareSymbolFont{bbold}{U}{bbold}{m}{n}
\DeclareSymbolFontAlphabet{\mathbbold}{bbold}

\newcommand{\md}[1]{  \left(\textnormal{mod}\ #1\right)}
 
\renewcommand{\P}{\mathbb{P}}

\renewcommand{\rm}{\mathrm}

\newcommand{\fa}{\mathfrak{a}}
\newcommand{\fb}{\mathfrak{b}}
\newcommand{\fc}{\mathfrak{c}}
\newcommand{\fd}{\mathfrak{d}}

\newcommand{\fp}{\mathfrak{p}}
\newcommand{\fm}{\mathfrak{m}}
\newcommand{\fn}{\mathfrak{n}}

\newcommand{\cP}{\mathcal{P}}

\newcommand{\Q}{\mathbb{Q}}
\newcommand{\F}{\mathbb{F}}

\newcommand{\N}{\mathbb{N}}
\newcommand{\C}{\mathbb{C}}
\newcommand{\R}{\mathbb{R}}

\newcommand{\Z}{\mathbb{Z}}

\renewcommand{\l}{\left}

\renewcommand{\r}{\right}
\renewcommand{\b}{\mathbf}
\renewcommand{\c}{\mathcal}
\renewcommand{\epsilon}{\varepsilon}
\renewcommand{\gcd}{\textrm{gcd}}

\renewcommand{\leq}{\leqslant}
\renewcommand{\geq}{\geqslant}
\renewcommand{\#}{\sharp}
\renewcommand{\gg}{\ggg}

\renewcommand{\ll}{\lll}

\newcommand\NN{\mathbb{N}}

\newcommand{\OO}{\mathcal{O}}

\newcommand{\ideals}{\mathcal{I}_K}
\newcommand{\id}[1]{\mathfrak{#1}}

\newcommand{\aaa}{\id{a}}

\newcommand{\n}{\mathfrak{N}}

\DeclareMathOperator*{\Osum}{\sum{}^*}

\newcommand{\beq}[2]
{
\begin{equation}
\label{#1}
{#2}
\end{equation}
}

\begin{document}

\title[
Twists of Hooley's $\Delta$-function over number fields
]
{
Twists of Hooley's $\Delta$-function over number fields
}

\author{Efthymios Sofos}
\address{
Max Planck Institute for Mathematics\\
Vivatsgasse 7
\\
Bonn
\\
53111
\\
Germany}
\email{sofos@mpim-bonn.mpg.de}

\begin{abstract} 
We prove tight estimates for averages of the twisted Hooley $\Delta$-function over arbitrary number fields.
\end{abstract}


\maketitle


\section{Introduction}
\label{intro}
In his memoir~\cite{hooldel}
Hooley studied the following function,
previously brought to attention by Erd{\H{o}}s,
\beq{def:del}{
\Delta(n):=
\max_{\substack{a\in \R}}
\#\big\{d \in \N: \mathrm{e}^a<d \leq \mathrm{e}^{a+1}, 
d \mid n\big\},
\
\
(n \in \N)
.}
He showed that its
average order
is genuinely
smaller than that of the divisor function, namely
\beq{def:delkaf}{
\frac{1}{x}
\sum_{n\leq x}\Delta(n)\ll (\log x)^{\frac{4}{\pi}-1}.}
This saving enabled him to provide
diverse applications
in topics related to
Diophantine approximation, 
divisor sums and problems of Waring's type.
Further applications 
were later
found
by Vaughan~\cite{MR803211},~\cite{MR826156},
for problems of Waring's type,
by Tenenbaum~\cite{MR897352}
in the topic of Diophantine approximation,
as well as for
Chebychev's problem on
the greatest prime factor of polynomial sequences
by Tenenbaum~\cite{MR1029397}.

The problem regarding the average of $\Delta$
was revisited by
Tenenbaum~\cite{MR814148},
who
established
a strong upper bound, 
with a special corollary 
that the exponent $\frac{4}{\pi}-1$ can be replaced 
by any positive constant. 
Specifically, letting  \[
\widehat{\epsilon}(x)
=\sqrt{\frac{\log \log \log (16+x)}{\log \log (3+x)}},
\] 
for any $x\geq 1$,
enables us to state his result,
namely
\beq{eq:first}
{
\frac{1}{x}
\sum_{n\leq x}\Delta(n)\ll (\log x)^{O(\widehat{\epsilon}(x))}.
}
In this paper we are interested in
generalisations
of 
functions 
similar to $\Delta$ 
over arbitrary number fields. 
Let $K$ 
be any number field with ring of integers denoted by $\OO_K$.
The symbol
$\ideals$  will be
reserved for
the monoid of non-zero integral ideals
of $\OO_K$,
while
$\n \fa=\#\OO_K/\fa$
will always refer to the ideal norm of $\aaa \in \ideals$. 
The generalisation of~\eqref{def:del} to $K$
 is given 
by
\[
\Delta_K(\fa):=
\max_{\substack{a \in \R}}
\#\big\{\fd \in \ideals:
\rm{e}^a  
<
\n\fd
\leq \rm{e}^{a+1} ,
\fd \mid \fa
\big\},
\
\
(\fa \in \ideals)
.\]  
The following result and its proof 
are entirely due to
Professors 
R\'{e}gis de la Bret\`eche
and
G\'{e}rald  Tenenbaum~\cite{letter}.
\begin{theorem}
\label{thm:main1} 
There exists a positive constant $c=c(K)$   
such that 
\[\frac{1}{x}
\sum_{\substack{ 
\n \fa  \leq x
}} 
\Delta_K(\fa)
\ll (\log x)^{c\widehat{\epsilon}(x)}
,\] 
where the implied constant is allowed to depend on $K$.
\end{theorem}
To show Theorem~\ref{thm:main1},
De la Bret\`eche
and
Tenenbaum
begin
by showing that
\[\frac{1}{x}
\sum_{\substack{ 
\n \fa  \leq x
}} 
\Delta_K(\fa)
\ll
\frac{1}{\log x} \sum_{n\leq x} \frac{\mu(n)^2}{n} \sum_{\n \fa=n} \Delta_K(\fa)
\]
in a way that is familiar to experts (see, for example~\cite{MR1618321}).
They now observe that for square-free $n$,
if there exists $d\in \N$ with $d \mid \n \fa$ then there exists a unique ideal $\fd$ with 
$\fd \mid n$, thus implying that 
$\Delta_K(\fa)=\Delta(n)$ in the last sum.
This shows that 
\[
\frac{1}{x}
\sum_{\substack{ 
\n \fa  \leq x
}} 
\Delta_K(\fa)
\ll
\frac{1}{\log x} \sum_{n\leq x} \frac{\mu(n)^2}{n} 
r_K(n)
\Delta(n)
,\]
where  
\beq{eq:rkm}
{
r_K(n):=\#\big\{\fd \in \ideals:\n\fd=n\big\},
\
\
(n \in \N)
.}
Next, choosing a monic irreducible polynomial
$f \in \Z[X]$ such that $K=\Q(\theta)$
for a root $\theta$ of $f$,
one can use the Dedekind--Kummer theorem 
to relate the function 
$r_K$ 
and $\varrho_f$,
defined by 
\[
\varrho_f(n):=\#\{t\in \Z/n\Z:f(t)\equiv 0\md{n} \}
,\]
in particular,
getting
\[
 \sum_{n\leq x} \frac{\mu(n)^2}{n} 
r_K(n)
\Delta(n)\ll
 \sum_{n\leq x} \frac{\mu(n)^2}{n} 
\varrho_f(n)
\Delta(n)
.\]
As a last step, they applied~\cite[Th.3]{MR1029397}
to establish
Theorem~\ref{thm:main1},
and, in addition, they showed that one can take 
any fixed constant $c>\sqrt{2}$.

Next, let 
$\psi_K$ be any quadratic Dirichlet character on $K$
and define
\beq{def:chi}{
\Delta_K(\fa;\psi_K):=
\sup_{\substack{a \in \R\\ 0\leq b\leq 1}}
\Big|
\sum_{\substack{
\fd\mid \fa
\\
\rm{e}^a  
<
\n\fd
\leq \rm{e}^{a+b} 
}}
\psi_K(\fd)
\Big|
,
\
\
(\fa \in \ideals).}
When 
$K=\Q$
this function was 
considered by
La Bret\`eche
and
Tenenbaum~\cite{MR2927803}, as well as 
by
Br{\"u}dern~\cite{bruedern}.
Their work culminates in the bound,
\beq{eq:tenenbaum}
{
\frac{1}{x}
\sum_{ n\leq x}\Delta_\Q(n;\psi_\Q)^2\ll (\log x)^{O(\widehat{\epsilon}(x))}
.}
In this paper
we
generalise this to any number field
by following the
arguments in~\cite{MR2927803}.
\begin{theorem}
\label{thm:main2}  
Let 
$\psi_K$ be a 
quadratic 
Dirichlet character defined
on any number field $K$.
There exists a positive constant $c=c(K,\psi_K)$  
such that 
\[
\frac{1}{x}
\sum_{\substack{
\n \fa  \leq x
}} 
\Delta_K(\fa;\psi_K)^2 
\ll 
(\log x)^{c\widehat{\epsilon}(x)}
,\]
where the implied constant is allowed to depend on $K$ and $\psi_K$. 
\end{theorem} 
Professors 
R\'{e}gis de la Bret\`eche
and
G\'{e}rald  Tenenbaum
have informed us that there may
be a neater 
way to prove
Theorem~\ref{thm:main2}
without simply translating their work~\cite{MR2927803}
to number fields, which is what we do in the present paper.
It must be noted that 
one cannot directly
deduce
Theorem~\ref{thm:main2} for all number fields $K$
from the work in~\cite{MR2927803}
in the same way
that 
Theorem~\ref{thm:main1}
was deduced from~\cite[Th.3]{MR1029397}.
This
 direct
deduction can only be made 
if $K$ is such that 
there exists a Dirichlet character $\chi$
defined in $\Q$ 
and satisfying
\beq{eq:kummerdedek}{
\fa \in \ideals
\Rightarrow
\psi_K(\fa)
=
\chi(\n \fa)
.}
However,~\eqref{eq:kummerdedek}
does
not hold for every $K$ and
$\psi_K$.
Indeed, let $K/\Q$
be a quadratic extension,
let $L/K$ be a quadratic extension
and define a quadratic character $\psi_K$
by defining 
$\psi_K(\fp)$ 
to be $1,-1$ or $0$ 
respectively
according to whether 
$\fp$ 
is split,
inert,
or ramifies in $L$.
If $p$ is a rational prime such that 
$p\OO_L$ is inert in $L$
and~\eqref{eq:kummerdedek}
holds,
then 
\[-1=\psi_K(\fp)
=\chi(\n \fp)
=\chi(p^2)=1
,\]
which
gives a contradiction.
We also provide a
number field $K$
for which the following
relaxed version of~\eqref{eq:kummerdedek}
fails:
assume
that there exists a 
function $g:\Z\to \C$ such that  
\beq{eq:kummerdedek2}{
\fp \in \ideals,
\n \fp
\text{ is a prime in } \Z
\Rightarrow
\psi_K(\fp)
=
g(\n \fp)
.}
Take 
\beq{ex:counter}{
K:=\Q(\sqrt{-1})
\ \text{ and } \
\psi_K(\fa):=\left(\frac{1+\sqrt{-1}}{\fa}\right)
,}
where $(\frac{1+\sqrt{-1}}{\cdot})$ is a
quadratic symbol in $K$.
For every rational 
prime $p$ 
with $p\equiv 5 \md{8}$
we have 
$p\OO_K=\fp \overline{\fp}$
for
a
prime ideal $\fp$ in $\ideals$.
Letting $\pm t$
be the two solutions of 
$t^2\equiv -1 \md{p}$,
the Dedekind--Kummer theorem
shows that 
\[
\{\psi_K(\fp),\psi_K(\overline{\fp})\}
=
\Bigg\{\!
\l(\frac{1+t}{p}\r),\l(\frac{1-t}{p}\r)
\!
\Bigg\}
,\]
where $(\frac{\cdot}{p})$
is the Legendre quadratic symbol in $\F_p$.
Note that $\n \fp=\n \overline{\fp}$
is a prime
in
$\Z$,
therefore, if~\eqref{eq:kummerdedek2}
holds, then $\psi_K(\fp)=\psi_K(\overline{\fp})$.
This means that 
\[
1=
\psi_K(\fp)\psi_K(\overline{\fp})
=\l(\frac{1+t}{p}\r)
\l(\frac{1-t}{p}\r)
=
\l(\frac{1+t^2}{p}\r)
=
\l(\frac{2}{p}\r)
.\]
Recall that $p\equiv 5\md{8}$,
therefore 
$(\frac{2}{p})=-1$. 
This
shows that~\eqref{eq:kummerdedek2} cannot hold for our $K$.

The proof of Theorem~\ref{thm:main2}, supplied in~\S\ref{s:nonprincipal},
follows closely
the approach in~\cite{MR2927803}
involving an
induction related to the number of prime ideal divisors of $\fa$.
There will only be 
minor modifications;
these 
are to take care of the fact that there may be several ideals
$\fd$
of a given norm in~\eqref{def:chi}.
Thus 
one has to deal with 
 short sums $\sum_{\fd}\psi_K(\fd)$
that 
contain an 
amount of terms
which is not necessarily bounded.
We have 
to show
that cancellation in this short sum
still occurs
in this situation
and it will turn out that 
the arguments of
La Bret\`{e}che and Tenenbaum~\cite{MR2927803}
are flexible enough to handle these issues when
suitably 
modified.

Let us finally
remark that interest in averages of
$\Delta$-functions
has lately spiked
due to applications to Manin's conjecture.
This is a central conjecture
in the area of Diophantine geometry,
introduced by Manin and his collaborators in~\cite{fmt},
whose
aim is to
provide a precise description
of the  
distribution of rational points
on Fano varieties.
However,
its status for surfaces has not yet
been fully
resolved. An important
r\^{o}le in proving the conjecture for Ch\^{a}telet surfaces
is assumed by the asymptotic estimation of divisor sums of the form
\[
\sum_{\substack{(s,t) \in \Z^2\\|s|,|t| \leq x}}
\sum_{\substack{d\in \N \\
d|F(s,t)}}
\psi_\Q(d)
,\]
as
$x\to +\infty$,
where  
$F\in \Z[s,t]$ is a separable
quartic form
and $\psi_\Q$ is a
quadratic
Dirichlet character. 
La Bret\`{e}che and Tenenbaum
used~\eqref{eq:tenenbaum}
to
handle these divisor sums
when $F$ is irreducible 
or a product of two irreducible quadratic forms
in~\cite{MR3103132},
which enabled them to prove
Manin's conjecture for two families of
Ch\^{a}telet surfaces.

In our  
 joint work~\cite{BS} with Browning, 
Theorems~\ref{thm:main1} and~\ref{thm:main2}
are used to
study divisor sums of
the shape
\beq{eq:gendiv}
{
\sum_{\substack{(s,t) \in \Z^2\\|s|,|t| \leq x}}
\Bigg(
\prod_{i=1}^n
\
\sum_{k_i|F_i(s,t)}
\l(\frac{G_i(s,t)}{k_i}\r)
\Bigg)
,}
where $F_i,G_i \in \Z[s,t]$ are appropriate
binary forms
with $\sum_{i=1}^n \deg(F_i)=4$
and
$\l(\frac{\cdot}{\cdot}\r)$ denotes the Jacobi symbol. 
As a byproduct
we provide
matching upper and lower bounds 
agreeing with Manin's conjecture
for every
Ch\^{a}telet surface 
and
every
quartic del Pezzo surface with a conic bundle structure over $\Q$. 
The example~\eqref{ex:counter}
can occur in the setting of 
quartic del Pezzo surfaces with a conic bundle structure.
Indeed, 
let
\[
\Phi_1(x_0,\ldots,x_4):=x_0x_1-x_2x_3,
\
\
\Phi_2(x_0,\ldots,x_4):=x_0^2+2x_1^2+x_2^2+x_3^2-x_4^2+x_1 x_3
,\]
and
consider the surface
$X\subset \P^4_\Q$ 
that is cut out by the system 
$\Phi_1=\Phi_2=0$.
The map $f:X\to \P^1_\Q$
given by 
\[
f(x):=
\begin{cases} 
[x_0,x_2] & \mbox{if } (x_0,x_2)\neq (0,0) \\
[x_3,x_1] & \mbox{if } (x_3,x_1)\neq (0,0)
\end{cases} 
\]
is a conic bundle morphism 
whose fibers are given by 
\[(a^2+b^2)x^2+
(a^2+ab+2b^2)y^2
=z^2
.\]
In the terminology of~\cite[\S 2]{BS}
we have 
$\theta=\sqrt{-1}, 
K=\Q(\sqrt{-1}),
G(s,t)=s^2+st+2t^2$,
as well as 
$g(x)=x^2+x+2$
and
$L=K(\sqrt{1+\sqrt{-1}})$.
Furthermore,~\cite[Lemma 2.2]{BS}
shows that  
\[h(s,t)=\sum_{k| s^2+t^2} 
\l(\frac{s^2+st+2t^2}{k}\r)
\]
can be written as a
sum
of the shape
\[\sum_{\fa \mid (s- \sqrt{-1} t)} \psi_K(\fa)
,\]
where $\psi_K$ is a quadratic character in $\Q(\sqrt{-1})$ that coincides with the one
in~\eqref{ex:counter}
because the character 
in~\eqref{ex:counter}
satisfies $\psi_K(\fp)=1$ 
if and only if $\fp$ splits in $K(\sqrt{1+\sqrt{-1}})$.
\begin{notation} 
The symbol $\fp$ will exclusively
refer throughout this paper to prime ideals in $\OO_K$
and
the residue degree of any
$\fp \subset \OO_K$
will be denoted by $f_\fp$.
We shall
make frequent use of
the multiplicative span of all linear prime ideals,
\beq{eq:span}{
\cP_K^\circ
=\{\fa\subset \OO_K: \fp\mid \fa \Rightarrow f_{\fp}=1\}
.}
The symbols 
$ \mu_K, \tau_K$
and
$\Lambda_K$
will be used for the 
M\"{o}bius, divisor and the von Mangoldt function 
on $\ideals$,
while $\omega_K$ will 
stand for
the number of distinct prime ideal divisors on $\ideals$.
Unless the contrary is explicitly stated, the implicit 
constants in Landau's $O$-notation and Vinogradov's $\ll$-notation are allowed to depend 
on $K$ and $\psi_K$ but no other parameters. 
Lastly, the notation
$f(x)\asymp g(x)$ will be taken to mean
$f(x)\ll g(x) \ll f(x)$.
\end{notation}
\subsection*{Acknowledgement}
We are grateful to G\'{e}rald Tenenbaum for his generous explanations.
We are furthermore indebted to 
R\'{e}gis de la Bret\`eche
and
G\'{e}rald  Tenenbaum
for providing us with the proof of Theorem~\ref{thm:main1}.

\section{Precursory maneuvers}
\label{s:prelim} We begin by
establishing the following property,
\beq{eq:property}
{
\fa,\fb \in \ideals \ \text{coprime}
\Rightarrow
\Delta_K(\fa\fb)\leq \tau_K(\fa)\Delta_K(\fb)
\
\text{and}
\
\Delta_K(\fa\fb;\psi_K)\leq \tau_K(\fa)\Delta_K(\fb;\psi_K)
.}
Indeed, any $\fd\mid \fa\fb$ can be written uniquely
as 
$\fd=\fd_1 \fd_2$,  where $\fd_1\mid \fa$, $\fd_2|\fb$.
Therefore 
\[
\sum_{\substack{
\fd\mid \fa\fb 
\\
\mathrm{e}^a 
<
\n\fd
\leq \mathrm{e}^{a+b} 
}}
\hspace{-0,3cm}
\psi_K(\mathfrak{d})
=
\sum_{\substack{
\fd_1\mid \fa
}}
\
\
\
\
\psi_K(\mathfrak{d_1}) 
\hspace{-0,5cm} 
\sum_{\substack{
\fd_2| \fb
\\  
\mathrm{e}^{a-\log \n\fd_1} 
<
\n\fd_2
\leq 
\mathrm{e}^{a+b-\log \n\fd_1}  
}}
\hspace{-0,5cm}
\psi_K(\mathfrak{d}_2)
\]
and a similar equality holds when $\psi_K$ is replaced by $1$.
The triangle inequality
ensures
the validity of~\eqref{eq:property}.
\begin{lemma}
\label{lem:simplif}
For any $W_0 \in \N$ and any
$f:\ideals \to \R_{\geq 0}$
define the pair of functions
\[
M(x;f):=1+\sup_{1\leq y \leq x}
\frac{1}{y}\sum_{\n\fa\leq y}f(\fa)
\]
and
\[ 
L(x,W_0;f):=1+\sup_{1\leq y \leq x}
\frac{1}{\log y}
\hspace{-0,3cm}
\sum_{\substack{\n\fa\leq y\\ \fa \in \cP_K^\circ \\ \gcd(\n\fa,W_0)=1}}
\frac{f(\fa)\mu_K(\fa)^2}{\n\fa}
.\]
If
there exists $t>0$ such that 
$
f(\fa\fb)\leq \tau_K(\fa)^t
f(\fb)
$
for all integral coprime ideals $\fa$, $\fb$
then for any $W_0 \in \N$
we have the following as $x\to \infty$,
\[
M(x;f)
\asymp_{t,W_0}
L(x,W_0;f)
.\]
\end{lemma}
\begin{proof}
Let us begin by showing that 
\beq{eq:pro1}{
\sum_{\substack{\n\fa\leq x}}
\frac{f(\fa)}{\n\fa}
\asymp_{W_0}
\hspace{-0,3cm}
\sum_{\substack{\n\fa\leq x\\ \fa \in \cP_K^\circ \\ 
\gcd(\n\fa,W_0)=1
}}
\frac{f(\fa)\mu_K(\fa)^2}{\n\fa}
.}
The non-negativity of $f$ makes the inequality $\gg$ clear.
To prove the remaining inequality we may factorise 
uniquely each $\fa\in \ideals$ as 
$\fa=\fb\fc\fd$,
where each prime ideal divisor $\fp$
of $\fb$ 
satisfies $\n\fp|W_0$
and each prime ideal factor of $\fa$ which is coprime to $W_0$ and has residue degree at least  $2$ divides $\fc$.
The property of $f$ stated in our lemma shows that
\[
\sum_{\substack{\n\fa\leq x}}
\frac{f(\fa)}{\n\fa}
\leq
\prod_{\substack{\n\fp\mid W_0}}
\l(\sum_{m=0}^{\infty}\frac{\tau_K(\fp^m)^t}{\n\fp^m}\r)
\prod_{\substack{\n\fp\leq x\\ f_\fp\neq 1}}
\l(\sum_{m=0}^{\infty}\frac{\tau_K(\fp^m)^t}{\n\fp^m}\r)
\sum_{\substack{\n\fd\leq x\\ \fd \in \cP_K^\circ \\ \gcd(\n\fd,W_0)=1}}\frac{f(\fd)}{\n\fd}
\]
and we see that the first term is $O_{t,W_0}(1)$.
Writing $\n\fp=p^g$ for a rational prime $p$
we see that 
the second product is
\[\ll
\prod_{2\leq g\leq [K:\Q]}
\prod_{p\leq x^{1/g}}
\l(\sum_{m=0}^{\infty}
\frac{(m+1)^t}{p^{gm}}\r)
\leq
\prod_{p\leq x }
\l(1+O_{t}\!\l(\frac{1}{p^2}\r)\r)^{[K:\Q]}
\ll
1 
.\]
It thus remains to show that 
\[
\sum_{\substack{\n\fd\leq x\\ \fd \in \cP_K^\circ \\ \gcd(\n\fd,W_0)=1}}\frac{f(\fd)}{\n\fd}
\ll
\sum_{\substack{\n\fd_1\leq x\\ \fd_1 \in \cP_K^\circ \\ 
\gcd(\n\fd_1,W_0)=1
}}\frac{f(\fd_1)\mu_K(\fd_1)^2}{\n\fd_1}
.\]
To this end, we may factorise uniquely each $\fd$ as $\fd_1\fd_2$ where $\fd_1$, $\fd_2$ are coprime,
$\fd_1$ is square-free
and $\fd_2$ is square-full.
We may thus infer that 
\[
\sum_{\substack{\n\fd\leq x\\ \fd \in \cP_K^\circ \\ \gcd(\n\fd,W_0)=1}}\frac{f(\fd)}{\n\fd}
\leq
\sum_{\substack{\n\fd_1\leq x\\ \fd_1 \in \cP_K^\circ \\ \gcd(\n\fd_1,W_0)=1}}\frac{f(\fd_1)\mu_K(\fd_1)^2}{\n\fd_1}
\sum_{\substack{\n\fd_2\leq x \\ \fp|\fd_2\Rightarrow \fp^2 |\fd_2}}\frac{\tau_K(\fd_2)^t}{\n\fd_2}
\]
and the proof of~\eqref{eq:pro1} is concluded by observing that 
the sum over $\fd_2$ is  
\[\leq 
\prod_{\n\fp\leq x}\l(
1+O_t\!\l(\frac{1}{\n\fp^2}\r)
\r)=O_t(1)
.\]
In light of~\eqref{eq:pro1}
it is sufficient for our lemma
to show that 
\beq{eq:simplif2}
{
M(x;f)
\asymp_{t}
1+
\sup_{1\leq y \leq x}
\frac{1}{\log y}\sum_{\substack{\n\fa\leq y}}
\frac{f(\fa)}{\n\fa}
.}
Abel's summation can be employed to prove the inequality $\gg$ in~\eqref{eq:simplif2}.
For the remaining inequality
let us factorise 
$\fa$ as $\fb\fc$ with $\fb$, $\fc$
coprime, 
$\fb$ square-free and 
$\fc$ square-full.
This yields
\[
\sum_{\n\fa\leq y}f(\fa)
\leq 
\sum_{\substack{\n\fc\leq y\\\fp|\fc\Rightarrow \fp^2|\fc}}\tau_K(\fc)^t
\sum_{\n\fb\leq y/\n\fc}f(\fb)\mu_K(\fb)^2
.\]
Therefore, if the following
holds
\beq{eq:br}{
\sum_{\n\fb\leq T}f(\fb)\mu_K(\fb)^2
\ll
T^{\frac{1}{2}}+
\frac{T}{\log T}
\sum_{\n\fb\leq T}\frac{f(\fb)}{\n\fb}
,}
then the required estimate~\eqref{eq:simplif2} becomes available thanks to
\begin{align*}
\sum_{\n\fa\leq y}f(\fa)
&\ll y\sum_{\substack{\n\fc\leq y\\\fp|\fc\Rightarrow \fp^2|\fc}}\frac{\tau_K(\fc)^t}{\n \fc}
\frac{1}{\log \frac{x}{\n\fc}}\sum_{\n\fb\leq x/\n\fc}\frac{f(\fb)\mu_K(\fb)^2}{\n\fb}
\\&\ll y \l(\sup_{1\leq y \leq x}\frac{1}{\log y}\sum_{\n\fb\leq y}\frac{f(\fb)}{\n\fb}\r)
\sum_{\substack{\n\fc\leq y\\\fp|\fc\Rightarrow \fp^2|\fc}}\frac{\tau_K(\fc)^t}{\n \fc}
\end{align*}
and
\[\sum_{\substack{\n\fc\leq y\\\fp|\fc\Rightarrow \fp^2|\fc}}\frac{\tau_K(\fc)^t}{\n \fc}
\ll
\prod_{\n\fp\leq y}\l(1+O_t\!\l(\frac{1}{\n\fp^2}\r)\r)\ll_t 1.\]
To prove~\eqref{eq:br}
we shall deploy the bound 
$f(\fb)\leq \tau_K(\fb)\ll \n\fb^{1/2}$ 
to obtain
\begin{align*}
\sum_{\n\fb\leq T}
f(\fb)\mu_K(\fb)^2
&=
\sum_{\substack{\n\fb\leq T^{1/4}}}
f(\fb)\mu_K(\fb)^2
+
\sum_{\substack{T^{1/4}<\n\fb\leq T }}
f(\fb)\mu_K(\fb)^2
\\
&
\ll  T^{1/2}
+
\sum_{\n\fb\leq T} 
f(\fb)\mu_K(\fb)^2
\frac{\log \n \fb }{\log T}
.\end{align*}
Employing
the identity $\mu_K(\fb)^2\log \n\fb=\sum_{\fb=\fc\fp}\log \n\fp$
allows us to bound the last sum by
\[
\frac{2}{\log T}
\sum_{\n\fc\leq T}
f(\fc)\mu_K(\fc)^2
\sum_{\n\fp\leq T/\n\fc} \log \n\fp
\ll
\frac{T}{\log T}
\sum_{\n\fc\leq T}
\frac{f(\fc)}{\n\fc}
,\]
where the prime number theorem for $K$
has been used. 
\end{proof} 
\begin{proposition}  
There exists a positive constant $c=c(K,\psi_K)$  
such that for any $W\in \NN$  we have
$$
\sum_{\substack{
\n \fa  \leq x,
\fa\in \cP_K^\circ  
\\
\gcd(\n \fa,W)=1
}}
\frac{\Delta_K(\fa;\psi_K)^2\mu_K(\fa)^2}{\n\fa}\ll  
(\log x)^{1+c\widehat{\epsilon}(x)}.
$$ 
The implied constant is allowed to depend on $K, W$ and the character $\psi_K$. 
\end{proposition}
\begin{proof} 
The claim
stems 
 from
Theorem~\ref{thm:main2} by taking $f(\fa)=\Delta_K(\fa;\psi_K)^2$, $t=2$ and $W_0=W$
in Lemma~\ref{lem:simplif}.
\end{proof}
Lemma~\ref{lem:simplif}
makes possible 
to deduce Theorem~\ref{thm:main2}
from the following claim:
\textit{
For any Dirichlet quadratic character $\psi_K$
there exist positive constants $c_2,z_2$
that depend only on $K$ and $\psi_K$  
such that} 
\beq{eq:easy2}{
\sum_{\substack{
\fp|\fa \Rightarrow \n\fp>z_2 \\
\fa\in \cP_K^\circ,
\n \fa  \leq x
}}\frac{\Delta(\fa;\psi_K)^2\mu_K(\fa)^2}{\n\fa}
\ll_{K,\psi_K}
(\log x)^{1+c_2\widehat{\epsilon}(x)}.
}

For
$\fa \in \ideals$,
$u \in \R$
and
$q\in \R_{\geq 1}$
we let 
\beq{def:term}{
\Delta_K(\fa;u)
:=  
\sum_{\substack{
\fd\mid \fa
\\
\rm{e}^u
<
\n\fd
\leq \rm{e}^{u+1} 
}}
1
\
\
\
\
\
\
\text{and}
\
\
\
\
\ 
\
M_q(\fa)
:=
\int_{-\infty}^{+\infty}
\Delta_K(\fa;u)^q
\rm{d}u
.}
\begin{lemma}
\label{lem:obvi}
For all $\fa\in \ideals$ and $q\in \N$ we have
$M_q(\fa)\leq \tau_K(\fa)^q$.
\end{lemma}
\begin{proof}
It is evident that 
$M_q(\fa)\leq \Delta_K(\fa)M_{q-1}(\fa)$,
hence the assertion can be validated 
by induction on $q$
upon noting that
$M_1(\fa)=\tau_K(\fa)$.
\end{proof}

\begin{lemma}
\label{lem:needed}
For each $\fb \in \ideals$ and positive integer $a$
we have  
\[
\sum_{\substack{\fd_1,\ldots,\fd_a \\ \fd_i|\fb  \\ \max \n\fb_i  < \rm{e} \min  \n\fb_i }}
\hspace{-0,7cm}
1
\leq 
2^{a+1}
M_a(\fb)
.\]
\end{lemma}
\begin{proof}
It is convenient to rewrite the last
summation condition as 
\[
2-\log\l(\frac{\max \n\fb_i  }{\min \n\fb_i  }\r)>1
,\]
hence, letting
$x^+=\max\{0,x\}$
for $x \in \R$,
we can bound the sum in the lemma by
\[
\sum_{\substack{\fd_1,\ldots,\fd_a \\ \fd_i|\fb}} \l(2-\log\l(\frac{\max \n\fb_i  }{\min \n\fb_i  }\r)\r)^+
.\]
Using 
the convention $(a,b]=\emptyset$
when
$a\geq b$ verifies
the succeeding identity for all $a,b \in \R$,
\[(b-a)^+=\int_{(a,b]}1\mathrm{d}u.\]
This provides the equality of the last sum with 
\[
\sum_{\substack{\fd_1,\ldots,\fd_a \\ \fd_i|\fb}} \int_{[\log \max \n\fb_i,2+\log \min \n\fb_i)}1\mathrm{d}u=
\int_{-\infty}^{+\infty} \bigg(\sum_{\substack{\mathrm{e}^{u-2}<\n\fd \leq \mathrm{e}^u
\\
\fd|\fb
}}1\bigg)^a\mathrm{d}u
,\]
which, upon decomposing the sum over $\fd$ as 
\[
\sum_{\substack{\mathrm{e}^{u-2}<\n\fd \leq \mathrm{e}^{u-1}\\\fd|\fb}}1
+
\sum_{\substack{\mathrm{e}^{u-1}<\n\fd \leq \mathrm{e}^u\\\fd|\fb}}1
,\]
leads to the desired
bound
\[
2^a
\int_{-\infty}^{+\infty}
\bigg(
\sum_{\substack{\mathrm{e}^{u-2}<\n\fd \leq \mathrm{e}^{u-1}\\\fd|\fb}}1
\bigg)
^a\mathrm{d}u
+
2^a
\int_{-\infty}^{+\infty}
\bigg(
\sum_{\substack{\mathrm{e}^{u-1}<\n\fd \leq \mathrm{e}^u\\\fd|\fb}}1
\bigg)^a\mathrm{d}u
,\]
that is clearly sufficient for the lemma.
\end{proof}
For integers $c$, $q$ in the range 
$1\leq c \leq q-1$
we can obtain via H\"older's inequality with exponents
$\frac{q-1}{q-c}$
and
$\frac{q-1}{c-1}$
the succeeding inequality
\[
M_c(\fb)=  
\int_{-\infty}^{+\infty} \Delta(\fb;u)^{\frac{q-c}{q-1}}   \Delta(\fb;u)^{\frac{q(c-1)}{q-1}} \mathrm{d}u
\leq 
M_1(\fb)^{\frac{q-c}{q-1}} 
M_q(\fb)^{\frac{c-1}{q-1}}
.\]
Using this for
$c=a$
and 
$c=q-a$
yields respectively
\beq{eq:2bu}
{
M_a(\fb)
\leq 
M_1(\fb)^{\frac{q-a}{q-1}} 
M_q(\fb)^{\frac{a-1}{q-1}}
\
\
\text{and}
\
\
M_{q-a}(\fb)
\leq   
M_1(\fb)^{\frac{a}{q-1}} 
M_q(\fb)^{\frac{q-a-1}{q-1}}
,}
an inequality that will be used later.

\section{The proof of Theorem~\ref{thm:main2}}
 \label{s:nonprincipal}
The bound~\eqref{eq:easy2}
will be proved 
by an induction process which is given
in \S\ref{s:indpro}.
The central result 
deployed
in this process is Proposition~\ref{p:central},
whose proof is
postponed until \S\ref{s:indpro2}.
\subsection{The induction process}
\label{s:indpro} 
Throughout \S\ref{s:nonprincipal} the positive real number $z_2=z_2(K,\psi_K)$
will be allowed to increase but it will 
be independent of the counting parameter
$x$. 
The Erd{\H{o}}s--Kac theorem for $K$ shows that 
the number of distinct prime ideal divisors
of a typical element $\fa \in \ideals$ is of size $\log \log \n\fa$,
thus suggesting  to consider the contribution of $\fa$
satisfying
 $\omega_K(\fa)>10 \log \log x$ in~\eqref{eq:easy2}. 
Using~\eqref{eq:property}
with $\fb=\OO_K$
we see that it is at most 
\[
\sum_{\substack{
\omega_K(\fa)>10 \log \log x \\
\n \fa  \leq x
}}\frac{\tau_K(\fa)^2\mu_K(\fa)^2}{\n\fa}
\leq  
\sum_{\substack{ 
\n \fa  \leq x
}}\frac{\tau_K(\fa)^2\mu_K(\fa)^2}{\n\fa}
\l(\frac{5}{2}\r)^{\omega_K(\fb)-10 \log \log x}
\]
and the inequality 
$10-10\log(\frac{5}{2})<1$
affirms the bound
\[
\l(\log x\r)^{-10\log(\frac{5}{2})}
\prod_{\n\fp\leq x}
\l(
1+\frac{10}{\n\fp}
\r)
\ll
\log x
.\]
This shows that~\eqref{eq:easy2}
stems from the estimate
\beq{eq:folws}{
\sum_{\substack{
\fp|\fa \Rightarrow \n\fp >z_2 \\
 \omega_K(\fa)\leq 10 \log \log x \\
\fa\in \cP_K^\circ,
\n \fa  \leq x
}}\frac{\Delta_K(\fa;\psi_K)^2\mu_K(\fa)^2}{\n\fa}
\ll_{K,\psi_K}
(\log x)^{1+c_2\widehat{\epsilon}(x)}
.}
We will soon replace the $\Delta$-term
by an expression 
involving an integral that approximates $\Delta_K(\fa;\psi_K)$.
The approximation can be performed when the divisors of $\fa$ are evenly spaced and we proceed 
by showing that the sum in~\eqref{eq:folws} can be restricted to $\fa$ 
with this  property.
For any $A>0$ we define   
$\c{E}(A)$ as the set of all $\fa \in \ideals$ for which there are 
distinct 
$\fd, \fd'$ with 
\[
\fd|\fa,
\fd'|\fa,
\ 
\
\
\n\fd\leq \n\fd'\leq \n\fd (1+(\log 2\n\fd)^{-A})
.\]
Assume that $A\geq 10$. Then each ideal counted in 
\beq{eq:shiuu}{
\sum_{\substack{
\fp|\fa \Rightarrow \n\fp >z_2 \\ 
\fa\in \cP_K^\circ
\cap \c{E}(A),
\n \fa  \leq x
}}\frac{\Delta_K(\fa;\psi_K)^2\mu_K(\fa)^2}{\n\fa}
}
is of the shape 
$\fa=\fd\fd'\fm$,
where $\fd,\fd',\fm$ are coprime in pairs and square-free
and satisfy
\[
\n\fd\leq \n\fd'\leq \n\fd (1+(\log 2\n\fd)^{-A})
.\]
Hence, by~\eqref{eq:property} with $\fb=\OO_K$,
the sum is bounded by
\[
\sum_{\substack{\fm,\fd \in \cP_K^\circ\\ \n \fm \n\fd \leq x }}
\frac{\mu_K(\fm)^2\mu_K(\fd)^2\tau_K(\fm)^{2}}{\n\fm \n\fd\tau_K(\fm)^{-2}}
\sum_{\substack{\fd' \in \cP_K^\circ,
\fp|\fa \Rightarrow \n\fp >z_2 \\
\n\fd\leq \n\fd'\leq \n\fd (1+(\log 2\n\fd)^{-A})
}}
\frac{\mu_K(\fd')^2\tau_K(\fd')^2}{\n\fd'}
.\]
Introducing the following
arithmetic function,
\[
f(d):=\sum_{\substack{\fd'\in \cP_K^\circ\\\n\fd'=d}}\mu_K(\fd')^2\tau_K(\fd')^2
,\]
allows us to bound the sum over $\fd'$ by  
\[
\n\fd^{-1}
\hspace{-0,3cm}
\sum_{\substack{
p|d \Rightarrow p >z_2 \\
\n\fd\leq d\leq \n\fd (1+(\log 2\n\fd)^{-A})}}f(d).
\]
Using~\cite[Th.1]{shiu} shows that the last expression is bounded by   $\n\fd^{-1}(\log \n\fd)^{-1-A}$
multiplied by a quantity that is bounded by 
\[\ll_A \exp\Big(\sum_{z_2<p\leq z_2+2\n\fd}f(p)/p\Big)\ll 
\exp\Big(\sum_{z_2<\n\fp\leq z_2+2\n\fd}4/\n\fp
\Big)\ll (\log \n\fd)^4.\]
We have thus shown that the sum in~\eqref{eq:shiuu} is
\[
\ll
\sum_{\substack{\fm \in \cP_K^\circ\\ \n \fm  \leq x }}\frac{\mu_K(\fm)^2\tau_K(\fm)^{2}}{\n\fm}
\sum_{\substack{\fd \in \cP_K^\circ\\ \n\fd \leq x }}\frac{\mu_K(\fd)^2\tau_K(\fd)^{2}}{\n\fd(\log \n\fd)^{A-3}}
\ll
(\log x)^4
\sum_{\substack{\fd \in \cP_K^\circ\\ \n\fd \leq x }}
\frac{\mu_K(\fd)^2\tau_K(\fd)^{2}}{\n\fd(\log \n\fd)^{A-3}}
.\]
By Abel's summation the sum over $\fd$ is $\ll (\log x)^{7-A}$, 
thus yielding
\beq{eq:shiuu2}{
A\geq 10 \Rightarrow
\sum_{\substack{
\fp|\fa \Rightarrow \n\fp >z_2 \\ 
\fa\in \cP_K^\circ\cap \c{E}(A),
\n \fa  \leq x
}}\frac{\tau_K(\fa)^2\mu_K(\fa)^2}{\n\fa}
\ll_A (\log x)^{11-A}
,} 
which reveals that, owing to~\eqref{eq:folws},
the next estimate is sufficient for the proof of~\eqref{eq:easy2}, 
\beq{eq:shiuu3}{
\sum_{\substack{
\n \fa  \leq x,
\fp|\fa \Rightarrow \n\fp >z_2 \\
\omega_K(\fa)\leq 10 \log \log x \\
\fa\in \cP_K^\circ, \fa \notin \c{E}(10) 
}}\frac{\Delta_K(\fa;\psi_K)^2\mu_K(\fa)^2}{\n\fa}
\ll_{K,\psi_K}
(\log x)^{1+c_2\widehat{\epsilon}(x)}
.}

The induction process that will enable us to prove~\eqref{eq:shiuu3}
requires that we are in possession 
of an ordering of the prime ideals $\fp\subset \OO_K$; thus we
 form the sequence 
$(\fp_i)_{i=1}^\infty$ 
such that  
\beq{eq:ordering}
{i<j \Rightarrow
\fp_i\neq \fp_j, \n\fp_i\leq \n\fp_{j}.
}
Prime ideals of the equal norm are allowed to be ordered
arbitrarily, but their ordering is
fixed once and for all.
Hence, for any $\fa$ we can 
set $i^+(\fa)=\max\{i\in \N:\fp_i \mid \fa\}$
and define
\[\fp^+(\fa):=\fp_{i^+(\fa)}.\]
Furthermore,
for each $r\in \N$ and square-free $\fa\in \ideals$,
we let $\fa_r:=\fa$ if $r\geq \omega_K(\fa)$.
If $r<\omega_K(\fa)$ holds then we
choose the first $r$
prime ideal divisors of $\fa$ according to the ordering above 
and let $\fa_r$ be their product.
Setting $r_x:=[10\log \log x]$ shows that the sum in~\eqref{eq:shiuu3} is
\[
\sum_{\substack{
\n \fa  \leq x,
\fp|\fa \Rightarrow \n\fp >z_2 \\
\omega_K(\fa)\leq 10 \log \log x \\
\fa\in \cP_K^\circ, \fa \notin \c{E}(10) 
}}\frac{\Delta_K(\fa_{r_x};\psi_K)^2\mu_K(\fa)^2}{\n\fa}
\leq
\sum_{\substack{
\n \fa  \leq x,
\fp|\fa \Rightarrow \n\fp >z_2 \\ 
\fa\in \cP_K^\circ 
}}\frac{\Delta_K(\fa_{r_x};\psi_K)^2\mu_K(\fa)^2}{\n\fa}
.\]
Letting for any 
$\fa \in \ideals$,
$a\in \R$
and
$b \in (0,1]$,
\beq{def:dd}{\Delta_K(\fa;\psi;a,b):=\Big|\sum_{\substack{\ \ \rm{e}^a <\n\fd\leq \rm{e}^{a+b}\\
\fd\mid \fa
}}\psi_K(\fd) \ \Big|}
sets the stage 
for the entrance of
the important entity
\beq{def:MMqq}
{
M_q(\fa;\psi_K)
:=
\int_{0}^1\int_{\R} \Delta_K(\fa;\psi;a,b)^q
\rm{d}a \rm{d}b
,
\
\
(\fa \in \ideals,
q \in \N)
.}
Let
$q_x:=[\sqrt{ r_x/(1+\log r_x)}]$
and
define for $r,q\in \N$ the average
\beq{def:cl}
{\c L(x):=
4^{\frac{r_x}{q_x}}
\log x
+\sum_{\substack{
\n \fa  \leq x,
\fa\in \cP_K^\circ 
\\
\fp|\fa \Rightarrow \n\fp >z_2
}}\frac{M_{2q_x}(\fa_{r_x};\psi_K)^{\frac{1}{q_x}}
\mu_K(\fa)^2}{\n\fa}
.}
The next lemma shows that Theorem~\ref{thm:main2} stems from
\beq{eq:shiuu4}{
\c L(x) 
\ll_{K,\psi_K}
(\log x)^{1+c_2\widehat{\epsilon}(x)}
.}
\begin{lemma}
\label{lem:e10}
For all $q \in \N$
and 
square-free $\fa\in \ideals$ 
with 
$\n\fa\leq x$ and 
$\fa\notin \c{E}(10)$
we
have 
\[
\Delta_K(\fa;\psi_K)^2
\leq 8^2
+2^{10}
(\log x)^{\frac{20}{q}}
M_{2q}(\fa;\psi_K)^{\frac{1}{q}}
.\]
\end{lemma}
\begin{proof} 
The lemma is valid if
$\Delta_K(\fa;\psi_K)< 8$, we may therefore assume
henceforth
that the opposite holds. Note that 
the definition of $\Delta_K(\fa;\psi_K)$ provides
$a_0\in \R,b_0 \in [0,1]$
such that 
\beq{eq:bta}{
|\Delta_K(\fa;\psi_K;a_0,b_0)|
\geq
\frac{1}{2}
\Delta_K(\fa;\psi_K)
.}
We bring into play the box
$\mathfrak{B} \subset \R^2$
given by 
\[
\Big(a_0,a_0+\frac{1}{8(\log 2x)^{10}}\Big)
\times
\Big(b_0,b_0+\frac{1}{8(\log 2x)^{10}}\Big),
\Big(a_0-\frac{1}{8(\log 2x)^{10}},a_0\Big)
\times
\Big(b_0-\frac{1}{8(\log 2x)^{10}},b_0\Big)
\]
respectively according to whether $b_0<\frac{1}{2}$ 
or not.
We choose to focus on the latter case; the former being treated similarly. 
For any $(a,b) \in \mathfrak{B} $
we have 
\beq{eq:empty}
{
|\Delta_K(\fa;\psi_K;a,b)
-
\Delta_K(\fa,\psi_K;a_0,b_0)
|
\leq 
\sum_{\substack{
\fd|\fa 
\\
\rm{e}^a\leq \n\fd \leq \rm{e}^{a_0}
}}1
+\sum_{\substack{
\fd|\fa 
\\
\rm{e}^{a+b}\leq \n\fd \leq \rm{e}^{a_0+b_0}
}}1
.}
If the first sum has more than one term then 
there exist
$
\fd\neq \fd' \in \ideals
$
with 
$\fd,\fd'|\fa$ and 
\[
\rm{e}^{a_0-1/8(\log 2x)^{10}}
\le
\n\fd \le \n\fd' \le 
\rm{e}^{u_0}
,\]
thus leading via 
$\n\fa\leq x$ and
$\rm{e}^z\le 1+2z$
(valid in the range $0<z<1$),
to
\[
  \n\fd' \le 
\n\fd
\
\rm{e}^{1/8(\log 2x)^{10}}
\leq 
\n\fd
\l(1+\frac{1}{4(\log 2x)^{10}}\r)
\leq \n\fd
\l(1+\frac{1}{(\log 2\n\fd )^{10}}\r)
,\]
which contradicts the assumption $\fa \notin \c{E}({10})$
of our lemma.
A similar argument shows that the second sum in~\eqref{eq:empty}
also contains at most one term, therefore invoking
$\Delta_K(\fa;\psi_K) \geq 8$
and~\eqref{eq:bta}
provides us with 
\[
\Delta_K(\fa;\psi_K;a,b)
\geq 
\frac{\Delta_K(\fa;\psi_K)}{2}
-2
\geq 
\frac{\Delta_K(\fa;\psi_K)}{4}
.\] 
This inequality immediately furnishes the required estimate
by restricting the range of integration in~\eqref{def:MMqq} to $\mathfrak{B} $.
\end{proof}
For positive integers $r$, $q$
and any $\sigma \in (0,\frac{1}{4}]$
we define the functions
\[
L^*_{r,q}(\sigma):=\frac{4^{\frac{r}{q}}}{\sigma}
+
\sum_{\substack{
\fa\in \cP_K^\circ \\
\fp|\fa \Rightarrow \n\fp >z_2 
}}\frac{M_{2q}(\fa_{r};\psi_K)^{\frac{1}{q}}
\mu_K(\fa)^2}{\n\fa^{1+\sigma}}
\]
and for $s\in \R_{\geq 1}$ 
we
let 
\[f(s):=\sqrt{\frac{s}{1+\log s}}.\]
Noting that 
$\c L(x)\leq \rm{e} L^*_{r_x,q_x}(1/\log x)$,
our aim now becomes
to prove that for all sufficiently small $\sigma>0$ we have 
\beq{eq:aim1}
{r\gg 1,
q=[f(r)]
\Rightarrow 
L^*_{r,q}(\sigma)
\ll
\frac{\rm{e}^{c_2 \sqrt{r \log r}}}{\sigma}
}
for some constant $c_2>0$ depending at most on $K$ and $\psi_K$.
Clearly, this is sufficient for verifying~\eqref{eq:shiuu4}.

The strategy for the proof of~\eqref{eq:aim1} 
is indirect and resembles a backwards induction process.
First, note that if the variable $r$ is replaced by any fixed integer constant $t$,
then for any $q$ we have $L^*_{t',q} \ll_{t'} 1/\sigma$. Indeed, 
using Lemma~\ref{lem:obvi}
in combination with the obvious bound $M_{2q}(\fa;\psi_K)\leq M_{2q}(\fa)$
furnishes
\beq{eq:triv1}
{L^*_{t',q}(\sigma) -\frac{4^{\frac{t'}{q}}}{\sigma}
\ll 
\sum_{\substack{
\fa \in \ideals
}}\frac{\tau_K(\fa_{t'})^2\mu_K(\fa)^2}{\n\fa^{1+\sigma}}
\leq
\sum_{\substack{
\fa \in \ideals
}}\frac{2^{2t'}}{\n\fa^{1+\sigma}}
\ll_{t'}
\zeta_K(1+\sigma)
\ll_t' \frac{1}{\sigma}
.}
It will therefore be advantageous to 
bound $L^*_{r,q}(\sigma)$
in terms of
$L^*_{r-1,q}(\sigma)$
for $r$ and $q$ in suitable ranges.
To this end we shall deploy the succeeding lemma, whose proof is postponed until \S \ref{s:indpro2}.
\begin{proposition}
\label{p:central} 
There exist positive constants $c_3,t',z_2,\sigma_0$ that depend at most on $K$ and
$\psi_K$,
such that for all integers $t,m$ in the range 
\[t'\leq t\leq 10 \log \frac{1}{\sigma},
\
\
m\leq 
\sqrt{t/(1+\log t)}
<m+2
,\]
and $\sigma \in (0,\sigma_0)$
we have 
\[
L^*_{t,m}(\sigma)
\leq
\rm{e}^{\frac{c_3}{m}}
L^*_{t-1,m}(\sigma)
.\]
\end{proposition}
To deduct~\eqref{eq:shiuu4}
from Proposition~\ref{p:central}
define for each integer
$\ell$
the following set,
\[\c{A}_\ell
:=\{
n \in \N:\ell\leq f(n)<\ell+1
\}
,\]
which furnishes the following
partition into disjoint sets
\[
\N\cap [t',r]
=\bigcup_{\ell\in \N}
\c{A}_\ell
.\]
Let $k:=\min\c{A}_\ell$.
It is easy to see that
$f(k+c\sqrt{k \log k})>f(k)+1$
holds for some large positive $c$
independent of $k$,
and therefore $\#\c{A}_\ell\leq c\sqrt{\ell \log \ell}$.
In addition, the definition of $k$ shows that $f(k-1)<\ell$ and therefore 
$\sqrt{k \log k} \ll \ell \log \ell$,
hence $\#\c{A}_\ell \leq c_4 q \log \ell$
for some absolute constant $c_4>0$. 
Furthermore,
$\c{A}_\ell$ will be empty unless
$\ell\leq q$.

It is now time to reveal our backwards induction process.
Whenever $n \in \c{A}_{q}$ 
we use Proposition~\ref{p:central}
with $t=n,m=q$
to reduce the value of $n$ from $r$ down to $\min \c{A}_{q}$.
This will come at a cost of
$\exp(\frac{c_3}{q}\#\c{A}_{q}) \leq q^{c_3 c_4}$.
At the end of this
section we shall prove that there exists a positive constant
$c_5=c_5(K,\psi_K)$ 
such that for all $\sigma>0$ sufficiently small we have 
\beq{eq:38}
{t'\leq t \leq 10 \log \frac{1}{\sigma},
\
m=[f(t)]
\
\
\Rightarrow
\
\
L^*_{t,m}(\sigma)
\leq 
m^{c_5}
 L^*_{t,m-1}(\sigma)
.}
When $n$ reaches $\min \c{A}_{q}$
we will use~\eqref{eq:38}
with $t=\min \c{A}_{q_x}$ and $m=q$.
We subsequently iterate the process by using Proposition~\ref{p:central} with $m=q-1$ and
$t=n$
for all $n \in \c{A}_{q-1}$.
We repeat this procedure going backwards until $\ell$ 
is small enough so that $t' \in \c{A}_\ell.$ 
The total cost will be 
\[
\ll
\prod_{\ell\leq q}\ell^{c_3 c_4+c_5} 
\leq
\rm{e}^{c_2 \sqrt{r \log r}} 
,\]
for some constant $c_2>0$ that depends at most on $K$ and $\psi_K$.
At the end of this process we shall be left with $L^*_{t',q}(\sigma)$
which can be estimated via~\eqref{eq:triv1},
thus concluding the proof
of~\eqref{eq:aim1}.

\begin{proof} [Proof of~\eqref{eq:38}]
Let us begin by
introducing the constants
$\eta_K:=\min\Big\{10^{-3},
2^{-\frac{3}{[K:\Q]}}\Big\}$
and $u_t:=\exp(70^t/t)$.
We shall make use of the set
$\c{D}_t$ that consists of all  square-free
$\fa \in \ideals$ for which there are 
distinct 
$\fd, \fd'$ satisfying
\[
\fd|\fa_t,
\fd'|\fa_t,
\ 
\
\
\n\fd\leq \n\fd'\leq \n\fd (1+\eta_K^t)
.\]
For each such
$\fa$ 
we can choose and fix square-free and coprime in pairs
$\fd_{\fa_t},
\fd'_{\fa_t},
\fm_{\fa_t} \in \ideals$
with 
$\fa=\fd_{\fa_t} \fd'_{\fa_t}\fm_{\fa_t}$ and 
$\fd_{\fa_t},\fd'_{\fa_t}$ being in the range designated above.
We may now deploy 
the inequality 
$\mu_K(\fa)^2
M_{2m}(\fa_{t};\psi_K)\leq
\mu_K(\fa)^2
\tau_{K}(\fa_t)^{2m}
=4^{tm}$
to
infer that for large enough $t\geq t'$
the contribution of
$\fa$ with $\n\fd_{\fa_t} \leq \exp(70^{t})$  towards 
$L^*_{t,m}(\sigma)-4^{\frac{t}{m}}/\sigma$ is at most
\[
4^t
\hspace{-0,5cm}
\sum_{\substack{ 
\fa \in \c{D}_t
\\
\eta_K^{-t}\leq
\n\fd_{\fa_t} \leq \exp(70^{t})
}}
\hspace{-0,3cm}
\frac{\mu_K(\fa)}{\n\fa^{1+\sigma}}
\leq 4^t 
\hspace{-0,4cm}
\sum_{\substack{ 
\fm \in \ideals
\\
\eta_K^{-t}\leq 
\n \fd  \leq \exp(70^{t})
}}
\hspace{-0,2cm}
\frac{1}{\n\fd^{1+\sigma}\n\fm^{1+\sigma}}
\hspace{-0,3cm}
\sum_{\substack{
\fd' \in \ideals \\
\n\fd\leq \n\fd'\leq \n\fd (1+\eta_K^t)
}}
\hspace{-0,3cm}
\frac{1}{\n\fd'^{1+\sigma}}
.\]
The estimate
$\sum_{\n\fd'\leq x}1=c_Kx+O_K(x^{1-\frac{1}{[K:\Q]}})$
shows that the sum over $\fd'$ is
\[
\leq \n\fd^{-1-\sigma}
\l(c_K \eta_K^t \n\fd+O_K(\n\fd^{1-\frac{1}{[K:\Q]}})
\r)
,\]
which provides the following bound,
\[
\ll
\frac{4^t}{\sigma}
\l(\eta_K^t
\sum_{\n\fd \leq \exp(70^{t})}\frac{1}{\n\fd}
+
\sum_{\n\fd\geq \eta_K^{-t}}\frac{1}{\n\fd^{1+[K:\Q]}}
\r)
\ll
\frac{4^t}{\sigma}
\l(\eta_K^t
{70^{t}}
+
{\eta_K^{t[K:\Q]}}
\r)
\ll
\frac{1}{\sigma 2^t}
.\]

Let us now focus on
the contribution of $\fa \in \c{D}_t$ with $\n\fd_{\fa_t}>
\exp(70^{t})
$.
The cardinality of
the prime ideal divisors of $\fa$ in the range $\n\fp \leq \exp(70^{t}/t)$,
henceforth denoted by $\omega(\fa;t)$,
cannot exceed $t$,
otherwise 
the first $t$ prime ideals dividing $\fa$ will have norm in that range, thus 
$
\n\fd_{\fa_t}
\leq 
\n\fa_t
\leq 
(\exp(70^{t}/t))^t
$,
which is  contradiction. 
In the case where
$\sigma>(32/9)t70^{-t}$
we see,
upon using $\zeta_K(1+\sigma)\ll \sigma^{-1}$,
that
the contribution of the ideals $\fa$
under consideration
towards 
$L^*_{t,m}(\sigma)-4^{\frac{t}{m}}/\sigma$  
is  at most
\[
4^t
\sum_{\substack{\fa \subset \OO_K}}
\frac{\mu_K(\fa)^2\n\fa_t^{\sigma/2}}
{\n\fa^{1+\sigma}{u_t}^{t\sigma/2}}
\ll
4^t
\sum_{\substack{\fa \subset \OO_K}}
\frac{\mu_K(\fa)^2}
{\n\fa^{1+\sigma/2}{u_t}^{t\sigma/2}}
\ll
\frac{(7/10)^t}{\sigma}
.\]
In the remaining case
$\sigma\leq (32/9)t 70^{-t}$
we set $v:=2/(\log 70)$
and bound the contribution by
\[4^t
\sum_{\substack{\fa \subset \OO_K}}
\frac
{\mu_K(\fa)^2v^{\omega(\fa;t)-t}}
{\n\fa^{1+\sigma}}
\ll
\frac{4^t}{v^{t}\sigma}
\prod_{\n\fp\leq \exp\{\sqrt{u_t}\}}(1+\n\fp^{-1})^{v-1}
,\]
which is again
$\ll (7/10)^t/\sigma$.
Thus far we have shown that 
\beq{eq:thusf}
{
L^*_{t,m}(\sigma)
\ll 
\frac{
4^{\frac{t}{m}}+(7/10)^{t}
}{\sigma}
+
\sum_{\substack{
\fa\in \cP_K^\circ, 
\fa \notin \c{D}_t
\\
\fp|\fa \Rightarrow \n\fp >z_2
}}\frac{M_{2m}(\fa_{t};\psi_K)^{\frac{1}{m}}
\mu_K(\fa)^2}{\n\fa^{1+\sigma}}
.}
Taking $z_2>2$
we see that each $\fa$ in the sum 
has odd norm, thus 
each element of the set 
\[S:=\big\{(a,b) \in (-1,0]\times (0,1): a+b \geq 0\big\}\]
satisfies  $\Z\cap(\rm{e}^u,\rm{e}^{a+b}]=\{1\}$ and therefore $\Delta_K(\fa;\psi_K;a,b)=1$.
Hence, for any $q\in \N$ we have 
$M_{2q}(\fa;\psi_K)\geq \rm{vol}(S)=1/2$.
We can now imitate the proof of
Lemma~\ref{lem:e10},
replacing $1/(\log 2x)^{10}$ by $\eta_K^t$, 
to prove that 
for all $q \in \N$
and 
square-free $\fa\in \ideals$ 
with  
$\fa\notin \c D_t$
we
have 
\[
\Delta_K(\fa_t;\psi_K)^2
\leq 8^2
+2^{10}
\eta_K^{-\frac{2t}{q}}
M_{2q}(\fa;\psi_K)^{\frac{1}{q}}
\le 2^{11}
\eta_K^{-\frac{2t}{q}}
M_{2q}(\fa_t;\psi_K)^{\frac{1}{q}}
.\]
Using this for $q=m-1$
in combination with  
\[
M_{2m}(\fa_t;\psi_K) 
\leq
\Delta_K(\fa_t;\psi_K)^2
M_{2m-2}(\fa_t;\psi_K) 
\]
leads to
\[M_{2m}(\fa_t;\psi_K)^{\frac{1}{m}}
\leq 
2^{\frac{11}{m}}
\eta_K^{-\frac{2t}{m(m-1)}}
M_{2m-2}(\fa_t;\psi_K)^{\frac{1}{m-1}}
.\]
The proof of~\eqref{eq:38} is concluded by
injecting
the last inequality into~\eqref{eq:thusf}
and
making use of $m \gg f(t)$ 
to derive 
$\eta_K^{-\frac{2t}{m(m-1)}}\leq m^{c_5}$
for some positive constant $c_5$ that depends at most on $K$ and $\psi_K$.
\end{proof}
\subsection{The proof of Proposition~\ref{p:central}}
\label{s:indpro2}
To relate
$L^*_{t,m}(\sigma)$
and
$L^*_{t-1,m}(\sigma)$
demands that we have an understanding
of the fluctuation of
$M_{2m}(\fa;\psi_K)^{\frac{1}{m}}$ 
as the number of prime ideal
divisors of $\fa$
varies.
To this end, we observe that for any
$\fa \in \ideals$
and prime 
$\fp$
we have
\[
\Delta_K(\fa\fp;\psi_K;a,b)=
\Delta_K(\fa;\psi_K;a,b)
+\psi_K(\fp)
\Delta_K(\fa,\psi_K;a-\log \n\fp,b)
.\]
For a positive integer $m$ we can raise 
to the power $2m$ to obtain 
\[
\Delta_K(\fa\fp;\psi_K;a,b)^{2m}=
\sum_{0\leq j \leq 2m} {{2m}\choose{j}}\psi(\fp)_K^{2m-j}
\Delta_K(\fa;\psi_K;a,b)^{j}
\Delta_K(\fa\fp;\psi_K;a-\log \n\fp,b)^{q-j}
.\]
Hence, letting for $\fa \in \ideals$, $w\in \R$,
$m \in \N$ and $ 0\leq j \leq m$,
\beq{def:Nq}
{N_{j,m}(\fa,w)
:=
\int_{0}^1\int_{\R} \Delta_K(\fa;\psi_K;a,b)^{j} \Delta_W(\fa,\psi;a-w,b)^{q-j} \rm{d}a \rm{d}b
}
and recalling~\eqref{def:MMqq}
we arrive at
\[
M_{2m}(\fa \fp;\psi_K)=
2M_{2m}(\fa;\psi_K)+
\sum_{1\leq j \leq 2m-1} {{2m}\choose{j}}\psi_K(\fp)^{j}
N_{j,2m}(\fa,\log \n\fp)
.\]
If $1< j < m-1$
we use $cd\leq \frac{c^2}{2}+\frac{d^2}{2}$ for 
\[
c=\Delta_K(\fa;\psi_K;a,b)^{j+1}\Delta_K(\fa;\psi_K;a-w,b)^{m-j-1},
d=\Delta_K(\fa;\psi;a,b)^{j}\Delta_K(\fa;\psi_K;a-w,b)^{m-j}
\]
to acquire \[N_{2j+1,2m}(\fa,w)\leq \frac{1}{2}N_{2j+2,2m}(\fa,w)+\frac{1}{2}N_{2j,2q}(\fa,w)\]
and
the inequality $cd\leq \frac{mc^2}{2}+\frac{d^2}{2m}$
yields
in like manner
\begin{align*}
N_{1,2m}(\fa,w)&\leq \frac{m}{2}N_{2,2m}(\fa,w)+\frac{1}{2m}M_{2m}(\fa;\psi_K),\\
N_{2m-1,2m}(\fa,w)&\leq \frac{m}{2}N_{2m-2,2m}(\fa,w) + \frac{1}{2m}M_{2m}(\fa;\psi_K).
\end{align*}
Putting everything together, we have 
\[
M_{2m}(\fa \fp;\psi_K)\leq 4M_{2m}(\fa;\psi_K)+W_{2m}(\fa,\fp)
,\]
where 
\beq{def:W}{W_{2m}(\fa,\fp):=\sum_{1\leq j \leq m-1} b_j {{2m}\choose{2j}} N_{2j,2m}(\fa,\log \n\fp)}
and the sequence given through 
\[b_j:=
\begin{cases} 
1+\frac{m}{2m-1}+\frac{2}{3}(m-2) & \mbox{if } j=1 \\
1+\frac{m}{2m-1} & \mbox{if } j=m-1 \\
 1+\frac{j}{2m-2j-1}+\frac{m-j}{2j+1} &\mbox{otherwise.}  
\end{cases} 
\]
satisfies 
$b_j\leq 1+\frac{2}{3}m$.

Assume that we are given $\fa \in \ideals$ with $\omega_K(\fa)>t-1$.
Then letting $\fp_{t}(\fa)$ be the $t$-th prime ideal factor of $\fa$ according to the ordering~\eqref{eq:ordering}
and using $(y_1+y_2)^{\frac{1}{m}}\leq y_1^{\frac{1}{m}}+y_2^{\frac{1}{m}}$, valid for 
$y_i \in \R_{\geq 0}$,
we
deduce 
\[M_{2m}(\fa_{t};\psi_K)^{\frac{1}{m}}\leq 4^{\frac{1}{m}}M_{2m}(\fa_{t-1};\psi_K)^{\frac{1}{m}}
+W_{2m}(\fa_{t-1},\fp_{t}(\fa))^{\frac{1}{m}}.\] This inequality is also valid 
if $\omega_K(\fa)\le t-1$, since in that case we have $\fa_{t}=\fa_{t-1}$. We obtain 
\[
L^*_{t,m}(\sigma)
\leq4^{\frac{1}{m}}
L^*_{t-1,m}(\sigma) 
+
\sum_{\substack{\fm \in \cP_K^\circ\\ \omega_K(\fm)=t-1 \\\fp|\fm \Rightarrow \n\fp >z_2}}
\sum_{\substack{\fp_j \in ^\circ_K\\ j>i^+(\fm) }}
W_{2m}(\fm,\fp_j)^{\frac{1}{m}}
\sum_{\substack{\fn_{t}=\fm \fp_j}}
\frac{\mu_K^2(\fn) }{\n\fn^{1+\sigma}}
.\] 
Each
ideal $\fn$ 
is
of the form $\fm \fp_j \fd$,
where $\fd$ is square-free and 
each prime divisor of $\fd$,
$\fp_i|\fd$
satisfies
$i>j$.
We can therefore deduce that
the sum over $\fn$ is 
 \[\ll 
\sum_{\substack{\fn_{t}=\fm \fp_j}}\frac{\mu_K^2(\fn) }{\n\fn^{1+\sigma}}
\ll
\n\fm \fp_j^{-1-\sigma}
\prod_{\n \fp>\n\fp_j}\l(1+\frac{1}{\n\fp^{1+\sigma}}\r),\] 
and, 
recalling that we denote
the Dedekind zeta function of $K$ by $\zeta_K$,
we deduce that the last product is 
\[\leq  \zeta_K(1+\sigma) \prod_{\n \fp\leq \n\fp_j}\l(1+\frac{1}{\n\fp^{1+\sigma}}\r)^{-1}
\ll\frac{1}{\sigma}
\prod_{\n \fp\leq \n\fp_j}\l(1-\frac{1}{\n\fp^{1+\sigma}}\r)
.\]
The inequality $\n\fp^{-\sigma}\geq 1-\sigma \log \n\fp$ and Mertens' theorem show that 
the inner product is 
\[\ll
\prod_{\n \fp\leq \n\fp_j}\l(1-\frac{1}{\n\fp}\r)
\exp\l(\sigma\sum_{\n \fp\leq \n\fp_j}\frac{\log \n\fp}{\n\fp}\r)
\ll
\frac{\n\fp_j^\sigma}{\log \n\fp_j}
,\]
thus showing that the sum over $\fn$ is $\ll_K\sigma^{-1} \ \n\fm^{-1-\sigma}(\n\fp_j\log \n\fp_j)^{-1}$. 
Letting
for $\fm \in \cP_K^\circ$,
\[
\c{A}_m(\fm):=
\sum_{\substack{\fp_j \in \cP_K^\circ
\\ j>i^+(\fm)}}\frac{W_{2m}(\fm,\fp_j)^{\frac{1}{m}}}{\n\fp_j \log \n\fp_j}
,\]
we have thus obtained
\beq{eq:mcbr}{
L^*_{t,m}(\sigma)-4^{\frac{1}{m}}
L^*_{t-1,m}(\sigma)
\ll
\frac{1}{\sigma}
\sum_{\substack{\fm \in \cP_K^\circ
\\ \omega_K(\fm)=t-1\\ \fp|\fm\Rightarrow \n\fp>z_2}}
\frac{\mu_K^2(\fm)}{\n\fm^{1+\sigma}}
\
\c{A}_m(\fm)
.}
Using H\"{o}lder's inequality with exponents $m,\frac{m}{m-1}$
we see that $\c{A}_m(\fm)$
is at most
\[
\Bigg(
\sum_{\substack{\fp_j \in \cP_K^\circ
\\ j>i^+(\fm)}}
\frac{W_{2m}(\fm,\fp_j)\log \n\fp_j}{\n\fp_j} 
\Bigg)
^{\frac{1}{m}}
\Bigg(\sum_{\substack{\fp_j \in \cP_K^\circ
\\ j>i^+(\fm)}}
\frac{1}{\n\fp_j \l(\log \n\fp_j\r)^{\frac{m+1}{m-1}}} 
\Bigg)^{\frac{m-1}{m}}
.\]
By the prime number theorem for $K$
and partial summation we infer that
with
$z:=\n\fp^+(\fm)$
the last sum is at most 
\[\leq 
\sum_{\substack{\n\fp >z/3}}
\frac{1}{\n\fp \l(\log \n\fp\r)^{\frac{m+1}{m-1}}}
\ll (\log z)^{-\frac{m+1}{m-1}}
\]
thus acquiring the validity of 
\beq{eq:acq}
{
\c{A}_m(\fm)
\ll
\Bigg(
\sum_{\substack{\fp_j \in \cP_K^\circ\\ j>i^+(\fm)}}
\frac{W_{2m}(\fm,\fp_j)\log \n\fp_j}{\n\fp_j} 
\Bigg)^{\frac{1}{m}}
(\log z)^{-\frac{m+1}{m-1}}
.}
For $\vartheta\in \R$
and
$\fa\in \ideals$ define 
\beq{eq:ramanujan}
{
\tau_K^*(\fa;\psi_K;\vartheta)
:=
\sum_{\substack{\fd\mid \fa}} \psi_K(\mathfrak{d})
\n\fd^{i\vartheta}
\ \ \
\text{and}
\ \ \
\tau_K^*(\fa;\psi_K)
:=\frac{1}{2\pi}
\int_{\R}
\frac{|\tau_K^*(\fa;\psi_K;\vartheta)|^2}{6+\vartheta^2}
\rm{d}\vartheta
.}
\begin{lemma}
\label{lem:plancherel}
For all $\fa\in \ideals$ we have $M_2(\fa;\psi_K)\leq \tau_K^*(\fa;\psi_K)$.
\end{lemma}
\begin{proof}
We start by using the following well-known
formula, 
valid for all $u,v, x\in \R$,
  \begin{equation}
    -\frac{1}{2\pi i}\int_{-\infty}^{+\infty}\frac{\rm{e}^{it(x-v)}-\rm{e}^{it(x-u)}}{t}\rm{d}t =
    \begin{cases*}
      1 & if $u<x$ and $x<v$, \\
      0 & if $x<u$ or $x>v$,  
    \end{cases*}
  \end{equation}
a proof of which can be found,
for example,
in~\cite[\S 5]{wiener}.
The substitution $t\mapsto 2\pi r$ gives
\[
\Delta_K(\fa;\psi_K;a,b) 
=\int_{-\infty}^{+\infty}
\l(\frac{1-\rm{e}^{-2\pi irb}}{2\pi i r}
\tau_K^*(\fa,\psi_K;2\pi r)\r)
\rm{e}^{-2\pi i ar}
\rm{d}r
\]
except when $a,b$ assume a finite set of values,
thus
Plancherel's theorem leads to
\[
\int_{-\infty}^{+\infty}
\Delta_K(\fa;\psi_K;a,b)^2\rm{d}a
=
\frac{1}{2\pi^2}
\int_{-\infty}^{+\infty} 
\frac{1-\cos(2\pi  rb)}{r^2}
\big|\tau_K^*(\fa,\psi_K;2\pi r)\big|^2
\rm{d}r
.\]
It can then be inferred 
from
$\int_{0}^1
(1-\cos(2\pi  rb))
\rm{d}b=1-\frac{\sin(2\pi r)}{2\pi r}$
that 
$2\pi^2M_2(\fa;\psi_K)$ equals
\[
\int_{-\infty}^{+\infty} 
\l(
1-\frac{\sin(2\pi r)}{2\pi r}
\r)
\frac{\big|\tau_K^*(\fa,\psi_K;2\pi r)
\big|^2}{r^2}
\rm{d}r
\]
and
the inequality 
$
1-\frac{\sin(2\pi r)}{2\pi r}
\leq \frac{4\pi^2 r^2}{3+2\pi^2r^2}$
furnishes the proof of our lemma.
\end{proof}
Define the arithmetic function $g:\N\to \Z$ though
$g(n):=\#\{\fp \subset \OO_K:\n\fp=n\}$
and note that the prime number theorem for $K$
provides a
positive
constant $\varkappa$ such that 
\beq{eq:pnt}{
\sum_{1\leq n \leq T}g(n)=\rm{li}(T)
+O\big(T\rm{e}^{-(\log T)^\varkappa}\big)
.}
Recall the definition of $M_q(\fa)$ in~\eqref{def:term}.
\begin{lemma}
\label{lem:24}
For all $\Xi
\geq 1$ and $m \in \N$
we have 
\begin{align*}
m^{-1}4^{-m}
\sum_{\substack{\n\fp>\Xi}}
\frac{W_{2m}(\fm,\fp)\log \n\fp}{\n\fp} 
&\ll
M_{2m}(\fm,\psi_K)^{\frac{m-2}{m-1}}
\tau_K^*(\fm;\psi_K)^{\frac{m}{m-1}} \\
&+
\rm{e}^{-(\log \Xi)^\varkappa}
4^m
M_{2m}(\fm)^{\frac{2m-2}{2m-1}}
\tau_K(\fm)^{\frac{2m}{2m-1}}
.\end{align*}
\end{lemma}
\begin{proof}
Using \eqref{def:W}
shows that the sum in our lemma is bounded by
\beq{eq:ourlem}{
\l(1+\frac{2}{3}m\r)
\sum_{1\leq j \leq m-1}
{{2m}\choose{2j}}
\sum_{\substack{\n\fp>\Xi}}
\frac{ N_{2j,2m}(\fm,\log \n\fp)
\log \n\fp}{\n\fp}
}
and 
the inner sum can be recast as
\[
\int_{0}^1\int_{\R} 
\Delta_K(\fm;\psi_K;a,b)^{2j} 
\bigg(
\sum_{\substack{\n\fp>\Xi}}
\frac{ \log \n\fp}{\n\fp}\Delta_K(\fm;\psi_K;a-\log \n\fp,b)^{2m-2j} 
\bigg)
\rm{d} a \rm{d}b
.\]
Letting $h:=2m-2j$ 
allows us to see that
the sum over $\fp$ equals
\[
\sum_{\substack{\fd_1,\ldots,\fd_h \in \ideals
\\ \fd_i|\fm}} \psi_K(\fd_1\cdots \fd_h)
\Osum_{n>\Xi}
g(n)
\frac{ \log n}{n}
,\]
where the sum $\Osum$ is
over integers 
$n$ satisfying the further condition
\[
a-\min \log \n\fd_i < \log n \leq a-\max \log \n\fd_i+b
.\]
This implies that
the sum contains no terms
unless  
$\max \n\fd_i < \rm{e}^b \min \n\fd_i $, in which case~\eqref{eq:pnt} and
Abel's summation provide the bound
\[
\ll
\int_{a-\min\log \n\fd_i}^{a+b-\log \max\n\fd_i}
\b{1}_{(0,\infty)}(t-\log \Xi) \rm{d}t+
\rm{e}^{-(\log \Xi)^\varkappa}
,\]
where
$\b{1}_{(0,\infty)}$ denotes the characteristic function of the positive real numbers. 
This confirms
\[
\sum_{\substack{\n\fp>\Xi}}
\frac{ \log \n\fp}{\n\fp}\Delta_K(\fm;\psi_K;a-\log \n\fp,b)^h
\ll
\int_{\R} \Delta_K(\fm;\psi_K;s,b)^h \rm{d}s
+
\rm{e}^{-(\log \Xi)^\varkappa}
\hspace{-0,8cm}
\sum_{\substack{\fd_1,\ldots,\fd_h\in \ideals,\fd_i|\fm
\\
\max \n\fd_i < \rm{e} \min \n\fd_i
}}
\hspace{-0,5cm}
1
\]
and according to Lemma~\ref{lem:needed}
the inner sum is 
$\ll 2^{h}M_h(\fm)$.
Therefore the sum over the prime ideals in~\eqref{eq:ourlem}
is  
\begin{align*}
&\ll
\int_{0}^1
\int_{\R} 
\Delta_K(\fm;\psi_K;a,b)^{2j} \rm{d}a
\int_{\R} \Delta_K(\fm;\psi_K;s,b)^{2m-2j} \rm{d}s
\rm{d}b
\\
&+
\rm{e}^{-(\log \Xi)^\varkappa}
4^{m-j}
M_{2j}(\fm)
M_{2m-2j}(\fm)
,\end{align*}
where we have used $|\psi_K(\fd)|\leq 1$
to dispense with the integration over $0\leq b \leq 1$ in the second term.
In virtue of~\eqref{eq:2bu}
and
Lemma~\eqref{lem:plancherel} one can show by following the argument involving H\"{o}lder's inequality
at the end of the proof of~\cite[Lem. 2.4]{MR2927803}
that the last expression is  
\[
\ll
M_{2m}(\fm;\psi_K)^{\frac{m-2}{m-1}}\tau^*(\fm;\psi_K)^{\frac{m}{m-1}}
+
\rm{e}^{-(\log \Xi)^\varkappa}
4^m
M_{2m}(\fm)^{\frac{2m-2}{2m-1}}
\tau_K(\fm)^{\frac{2m}{2m-1}}
,\] which, in view of $
\sum_{1\leq j \leq m-1}
{{2m}\choose{2j}}= 4^m-2$, 
finishes our proof.
\end{proof} 
We may now deploy the bound supplied by Lemma~\ref{lem:24} 
in conjunction with~\eqref{eq:acq} to obtain
\beq{eq:right}
{
\c{A}_m(\fm)
\ll
\c{B}_m(\fm)
+\c{C}_m(\fm)
,}
where
\[
\c{B}_m(\fm):= 
\frac{
M_{2m}(\fm;\psi_K)^{\frac{m-2}{m(m-1)}}
\tau_K^*(\fm;\psi_K)^{\frac{1}{m-1}}}{(\log z)^{\frac{m+1}{m-1}}}
\]
and
\[
\c{C}_m(\fm):=
\frac{
M_{2m}(\fm)
^{\frac{2m-2}{m(2m-1)}}
\tau_K(\fm)^{\frac{2}{2m-1}}}
{\rm{e}^{\frac{1}{m}(\log z)^\varkappa}}
.\]
Alluding to Lemma~\ref{lem:obvi}, the 
term $\c C_m(\fm)$
makes the following
contribution towards~\eqref{eq:mcbr},
\[
\ll
\frac{1}{\sigma}
\sum_{\substack{\fm \in \cP_K^\circ\\ 
\omega_K(\fm)=t-1
\\
\fp|\fm\Rightarrow \n\fp>z_2}}
\frac{\mu_K^2(\fm)\tau_K(\fm)^{2}}{\n\fm^{1+\sigma}}
\exp\l(-\frac{1}{m}(\log \n\fp^+(\fm))^\varkappa
\r)
.\]
Each $\fm$ above is the product of
$\fp^+(\fm)$
and $t-2$ prime ideals 
$\fp_i$ 
with
$\n\fp_i\leq \n\fp^+(\fm)$.
Taking into account the possible permutations of the ideals $\fp_i$
shows that
the sum over $\fm$ is 
\[
\ll
\frac{1}{(t-2)!}\sum_{\n\fp>z_2}
\frac{\exp\l(-\frac{1}{m}(\log \n\fp)^\varkappa\r)
}{\n\fp}
\l(\sum_{\n\fp_i\leq \n\fp}
\frac{4}{\n\fp_i}
\r)^{t-2}
.\] 
The sum over $\fp_i$ is at most $4\log_2\n\fp+O(1)$,
hence
using the inequality $\exp(-x)\leq \frac{\ell!}{x^\ell} $, valid for all $x\geq 0,\ell \in \N$,
we obtain that the expression above is bounded by 
\[
\ll
\frac{4^t\ell! m^\ell}{(t-1)!}
\sum_{\fp>z_2}\frac{1}{\n\fp}\frac{(\log \log \n\fp)^{t-2}}{(\log\n\fp)^{\varkappa \ell}}
.\]
We may suppose that $t'$ 
satisfies
$(\frac{\varkappa t'}{5}-1)>1$
and $t'>5$, so that upon
choosing $\ell:=[\frac{t}{5}]$ 
we see that the sum is 
\[
\ll
\int_{z_2}^\infty\frac{(\log \log u)^t}{u(\log u)^{t/5}}\rm{d}u
=
\Big(\frac{\varkappa t}{5}-1\Big)^{-t}
\int_{
(\frac{\varkappa t}{5}-1)
\log \log z_2}^{\infty}
\frac{v^t}{\rm{e}^v}
\rm{d}v
\leq 
\Big(\frac{\varkappa t}{5}-1\Big)^{-t}
t!
.\]
Therefore,
using $\log m=\frac{1}{2}
\log t+O(\log \log t)$ 
(which is implied by the assumptions of Proposition~\ref{p:central}),
as well as $\log n!=n \log n+O(n)$,
we see that
the contribution
of the entity
$\c{C}_m(\fm)$
towards~\eqref{eq:mcbr}
is
\[
\ll
\sigma^{-1}
\frac{4^t\ell!m^\ell 10^{t}}{\lambda_0^\ell t^{t}} 
\leq 
\sigma^{-1}
\exp(-\frac{7}{10}t\log t+O(t \log \log t))
\ll \frac{1}{\sigma
(t!)^{2/3}}
.\]

We now turn our attention to the contribution of 
$\c{B}_m(\fm)$
to~\eqref{eq:mcbr}.
It is at most
\beq{eq:wis}{
\sum_{\substack{\fm \in \cP_K^\circ
\\ \omega_K(\fm)=t-1\\ \fp|\fm\Rightarrow \n\fp>z_2}}
\frac{M_{2m}(\fm;\psi_K)^{\frac{m-2}{m(m-1)}}\tau_K^*(\fm;\psi_K)^{\frac{1}{m-1}}}{(\log \n\fp^+(\fm))^{\frac{1}{m}}}
\frac{\mu_K^2(\fm)}{\n\fm^{1+\sigma}} 
\frac{1}{\sigma \log \n\fp^+(\fm)}  
.}
For $\fm$ as in the sum above we let $S(\fm)$
be the set of square-free elements $\fn \in \cP_K^\circ$
that are
divisible by $\fm$ 
with the further property that 
any prime ideal  $\fp_i|\fn$ with $\fp_i \nmid \fm$ satisfies 
$i>i^+(\fm)$.
These ideals enjoy the property
$\fn_{t-1}=\fm$ and therefore 
\[\sum_{\substack{\fn \in \cP_K^\circ \\\fn_{t-1}=\fm}}\frac{\mu_K(\fn)^2}{\n\fn^{1+\sigma}}\geq\sum_{\substack{\fn \in S(\fm)\\}}\frac{\mu_K(\fn)^2}{\n\fn^{1+\sigma}}\geq \frac{\mu_K(\fm)^2}{\n\fm^{1+\sigma}}
\prod_{\substack{i>i^+(\fm)\\ \fp_i \in \cP_K^\circ}}
\l(1+\frac{1}{\n\fp_i^{1+\sigma}}\r).\]
The effect of primes $\fp_i$ with residue degree more than $1$ is bounded by a constant depending only on $K$,
thus the product is 
\[
\gg 
\zeta_K(1+\sigma)
\prod_{i \le i^+(\fm)}\l(1-\frac{1}{\n\fp_i^{1+\sigma}}\r)
\gg
\frac{1}{\sigma}
\exp\l(-\sum_{i \le i^+(\fm)}\frac{1}{\n\fp^{1+\sigma}}\r)
,\]
which by the Mertens theorem for $K$ is 
\[
\gg \frac{1}{\sigma}\exp\l(-\sum_{i \le i^+(\fm)}\frac{1}{\n\fp}\r)
\gg \frac{1}{\sigma \log \n\fp^+(\fm)}
.\]
We deduce that the sum in~\eqref{eq:wis} is 
\[
\ll
\sum_{\substack{\fm \in \cP_K^\circ
\\ 
\omega_K(\fm)=t-1
\\
\fp|\fm\Rightarrow \n\fp>z_2
}}
\frac{M_{2m}(\fm;\psi_K)^{\frac{m-2}{m(m-1)}}
\tau_K^*(\fm;\psi_K)^{\frac{1}{m-1}}}{(\log \n\fp^+(\fm))^{\frac{1}{m}}}
\sum_{\substack{\fn \in \cP_K^\circ
\\
\fn_{t-1}=\fm}}
\frac{\mu_K(\fn)^2}{\n\fn^{1+\sigma}}
.\]
Observe that for each $\fn$ in
the inner sum
we have 
$\omega_K(\fn)\geq \omega_K(\fm)=t-1$
and therefore the double sum may be
reshaped into
\[
\sum_{\substack{\fn\in \cP_K^\circ\\ \omega_K(\fn)\geq t-1}}
\l(
\frac{\mu_K(\fn)^2M_{2m}(\fn_{t-1};\psi_K)^{\frac{m-2}{m(m-1)}}
}{\n\fn^{\frac{(1+\sigma)(m-2)}{m-1}}}
\r)
\l(
\frac{\mu_K(\fn)^2\tau_K^*(\fn_{t-1};\psi_K)^{\frac{1}{m-1}}}
{\n\fn^{\frac{1+\sigma}{m-1}}(\log \n\fp^+(\fn_k))^{\frac{1}{m}}}
\r)
.\]
Assuming that $t'$ is large enough so that $m>2$ we can 
gain the succeeding bound via a use of 
H\"{o}lder's inequality with exponents 
$\frac{m-1}{m-2}$,
$m-1$,
\[
\ll
L^*_{t-1,m}(\sigma)^{\frac{m-2}{m-1}}
\c{D}_{t-1,m}(\sigma)^{\frac{1}{m-1}}
,\] 
where for $t,m$ positive integers and $\widehat\sigma \in (0,\frac{1}{4})$ we have defined 
\[
\c{D}_{t-1,m}(\widehat{\sigma})
:=
\sum_{\substack{\fn\in \cP_K^\circ\\ \omega_K(\fn)\geq t-1}} 
\frac{\mu_K(\fn)^2
\tau_K^*(\fn_k;\psi)}{\n\fn^{1+\widehat\sigma}(\log \n\fp^+(\fn_{t-1}))^{\frac{m-1}{m}}}
.\]
To estimate 
$\c{D}_{t-1,m}(\widehat{\sigma})$
we shall need the following lemma.
\begin{lemma}
\label{lem:2bmv} 
For $\vartheta,\Gamma \in (0,\infty)$ define  
\[S(\Gamma;\vartheta):=
\sum_{\substack{\n\fp\leq \Gamma
\\ \fp \in \cP_K^\circ
}}\frac{|1+\psi_K(\fp)\n\fp^{i\vartheta}|^2}{\n\fp}
.\]
There exists a constant $B=B(K,\psi_K)$ such that
the following holds uniformly in $\vartheta$,
\[
S(\Gamma;\vartheta)\leq 
\begin{cases} 
2\log\l(1+|\vartheta|\log \Gamma\r)+2\log\l(\frac{\log \Gamma}{1+|\vartheta|\log \Gamma}\r) +O(1), & \mbox{if } 0<|\vartheta|\leq 1 \\
2\log \log \Gamma+B\log \log (2+|\vartheta|), &\mbox{otherwise,}  
\end{cases} 
\]
where the implied constants are independent of $\Gamma$ and $\vartheta$.
\end{lemma}
\begin{proof}
For $j=1,-1$ we let 
\[
S_j(\Gamma;\vartheta)
:=
\sum_{\substack{\n\fp\leq \Gamma\\ \fp \in \cP_K^\circ,
\psi_K(\fp)=j}}\frac{|1+\psi_K(\fp)\n\fp^{i\vartheta}|^2}{\n\fp}
\]
so that 
$S(\Gamma;\vartheta)=\sum_{j\in \{1,-1\}}S_j(\Gamma;\vartheta)+O(1)$.
Introduce
the functions $g_j:\N\to\Z_{\geq 0}$ via
\[
g_j(n):=\#\{\fa\in \cP_K^\circ:\n\fa=n,\psi_K(\fa)=j\}
\]
and
note that the condition 
$\fp \in \cP_K^\circ$ forces $\n\fp$ to be a rational prime. 
The quantitative version of Chebotarev's theorem
provides 
positive constants $c',\eta'$ such that 
\[
\sum_{p\leq \Gamma} g_j(p) 
=\sum_{\substack{\n\fp\leq \Gamma \\ \fp \in \cP_K^\circ,
\psi_K(\fp)=j}}1
=
\frac{\rm{li}(\Gamma)}{2}+O(\Gamma\rm{e}^{-c'(\log \Gamma)^{\eta'}})
,\]
due to the standard bound
\[
\sum_{\substack{\n\fp \leq \Gamma\\f_{\fp}>1}}1\ll_{\epsilon,K}  {\Gamma^{\frac{1}{2}+\epsilon}}
,\]
valid for all $\epsilon>0$. 
Hence, in the notation of~\cite[Lem.III.4.13]{MR3363366},
we can use $h(r):=|1+\rm{e}^{ir}|^2$ and directly
modify its proof to show
that 
for each $w<\Gamma$
the following equality holds uniformly in $\vartheta\neq 0$,
\[
\sum_{\substack{w<\n\fp\leq \Gamma
\\
\fp \in \cP_K^\circ
\\ \psi_K(\fp)=j}}
\frac{|1+j\n\fp^{i\vartheta}|^2}{\n\fp}
=
\sum_{w<p\leq \Gamma}
|1+j p^{i\vartheta}|^2
\frac{g_j(p)}{p}
=
\log\l(\frac{\log \Gamma}{\log w}\r)
+O\l(
\frac{1}{|\vartheta|\log w}+\frac{1+|\vartheta|}{\rm{e}^{c^{''}(\log w)^{\eta'}}}
\r)
,\] owing to $\overline{h}=2$ for our choice of $h$.
This equality is parallel to~\cite[Eq.(3.16)]{MR2927803},
our proof can thus be concluded as the one of~\cite[Lem.2.5]{MR2927803} 
by using it for suitable parameters $w$ according to the value of $\vartheta$ in relation to $\Gamma$.
\end{proof}

\begin{lemma}
\label{lem:chicken}
For all 
$t$,
$m$ as in Proposition~\ref{p:central}
and $\widehat\sigma\in \R \cap(0,\frac{1}{4})$, we have 
\[
\c{D}_{t-1,m}(\widehat\sigma)
\ll 
\frac{t-1}{\widehat\sigma\l(1-\frac{1}{2m}\r)^{t-1}}
.\]
\end{lemma}
\begin{proof}
Fix an element $\fm\in \cP_K^\circ$.
The integral ideals $\fn$ in $\c{D}_{t-1,m}(\widehat\sigma)$ 
with $\fn_{t-1}=\fm$ 
are of the shape
$\fn=\fm\fd$, where $\fd$ is square-free
and furthermore if
$\fp_i|\fd$ then $i>i^+(\fm)$.
Hence, 
\[
\c{D}_{t-1,m}(\widehat\sigma)
\leq
\sum_{\substack{
\fm \in \cP_K^\circ
\\ \omega_K(\fm)= t-1}}
\frac{\mu_K(\fm)^2\tau_K^*(\fm;\psi_K)}{\n\fm^{1+\widehat\sigma}(\log \n\fp^+(\fm))^{\frac{m-1}{m}}}
\prod_{i>i^+(\fm)}\l(1+\frac{1}{\n\fp_i^{1+\widehat\sigma}}\r)
.\]
The last product is 
\[
\ll \zeta_K(1+\widehat\sigma)\exp\Big(-\sum_{i\leq i^+(\fm)}\n\fp_i^{-1-\widehat\sigma}\Big)
\ll \widehat\sigma^{-1}\exp\Big(-\sum_{i\leq i^+(\fm)}\n\fp_i^{-1-\widehat\sigma}\Big)
.\]
The 
inequality 
$
1-\n\fp_i^{-\widehat\sigma}
\leq \widehat\sigma \log \n\fp_i
$
reveals that  
\[
\sum_{i\leq i^+(\fm)}\n\fp_i^{-1-{\widehat\sigma}}
\geq 
\sum_{i\leq i^+(\fm)}\frac{1}{\n\fp_i}
-{\widehat\sigma}
\sum_{i\leq i^+(\fm)}\frac{\log \n\fp_i}{\n\fp_i}
,\]
which, by the prime number theorem for $K$, is 
$
\ll\log \log \n\fp^+(\fm)-{\widehat\sigma} \log \n\fp^+(\fm)
$.
We can therefore bound 
$\c{D}_{t-1,m}(\widehat\sigma)$ by 
\[
\ll
\frac{1}{\widehat\sigma}
\sum_{\substack{\fm \in \cP_K^\circ\\ \omega_K(\fm)= t-1}}
\frac{\mu_K(\fm)^2\tau_K^*(\fm;\psi_K)}{\n\fm(\log \n\fp^+(\fm))^{\frac{2m-1}{m}}}
\frac{\n\fp^+(\fm)^{\widehat\sigma}}{\n\fm^{\widehat\sigma}}
\]
and, letting
\[
T:=\sum_{\substack{\fm \in \cP_K^\circ \\ \omega_K(\fm)= t-1}}
\frac{\mu_K(\fm)^2\tau_K^*(\fm;\psi_K)}{\n\fm(\log \n\fp^+(\fm))^{\frac{2m-1}{m}}}
,\]
allows us to deploy
the inequality 
$\n\fp^+(\fm)\leq \n\fm $ to infer that 
$\c{D}_{t-1,m}(\widehat\sigma) \ll
T/\widehat\sigma$.
For $\vartheta \in \R$ let
\[T(\vartheta):= 
\sum_{\substack{\fm\subset \cP_K^\circ
\\ \omega_K(\fm)= t-1}}
\frac{\mu_K(\fm)^2|\tau_K^*(\fa;\psi_K;\vartheta)|^2}{\n\fm(\log \n\fp^+(\fm))^{\frac{2m-1}{m}}}
.\]
Note that
alluding to~\eqref{eq:ramanujan} and using
$\tau^*(\fm,\psi;-\vartheta)=\overline{\tau^*(\fm,\psi;\vartheta)}$ provides us with 
\[
T
\ll \int_0^\infty
T(\vartheta)
(1+\vartheta^2)^{-1}
\mathrm{d}
\vartheta
.\]
Denote
$\fp'=\fp^+(\fm)$.
Each ideal $\fm$ in $T(\vartheta)$
is the product of $\fp'$ and $k-1$ different prime ideals $\fp_i \in \cP_K^\circ$ 
that satisfy $i<i^+(\fm)$. 
Therefore 
\[
T(\vartheta)
\ll 
\frac{1}{(t-2)!}
\sum_{\fp' \in \cP_K^\circ}
\frac{|\tau_K^*(\fp';\psi;\vartheta)|^2}{\n\fp'(\log \n\fp')^{\frac{2m-1}{m}}}
\l(\sum_{\substack{\n\fp\leq \n \fp'\\\fp \in \cP_K^\circ}}
\frac{|\tau_K^*(\fp;\psi_K;\vartheta)|^2}{\n\fp}
\r)^{t-2}
\]
and using Lemma~\ref{lem:2bmv}
allows us to follow 
the arguments proving~\cite[Eq.(2.25),(2.26)]{MR2927803} 
to acquire the bound
\[
\int_0^1
T(\vartheta)\rm{d}\vartheta
\ll
\frac{2^t}{(t-2)!}
\bigg\{
\frac{(t-2)!}{(\frac{2m-1}{m}+1)^{t-1}}+
\frac{(t-1)!}{(\frac{2m-1}{m})^{t-1}}+
\frac{(t-1)!}{(\frac{2m-1}{m}+1)^{t-1}}
\bigg\}
\ll
\frac{t-1}{(1-\frac{1}{2m})^{t-1}}
.\]
In the remaining range, $\vartheta>1$, 
one can conjure up Lemma~\ref{lem:2bmv} and the proof of~\cite[Eq.(2.27)]{MR2927803} 
to deduce the estimate 
\[
T(\vartheta)\ll 
\l(1-\frac{1}{2m}\r)^{-(t-1)}
\{\log(2+\vartheta)\}^{\frac{B}{1-\frac{1}{2m}}} 
,\]
which, in light of
\[\int_{1}^\infty
\{\log(2+\vartheta)\}^{\frac{B}{1-\frac{1}{2m}}} 
\vartheta^{-2}
\rm{d}\vartheta
\ll_B 1
,\]
is sufficient for our lemma.
\end{proof}
Assorting all appropriate estimates obtained
so
far
validates 
\beq{eq:fopf}
{
L^*_{t,m}(\sigma)-4^{\frac{1}{m}}L^*_{t-1,m}(\sigma)
\ll
\frac{L^*_{t-1,q}(\sigma)^{\frac{m-2}{m-1}}(t-1)^{\frac{1}{m-1}}}{\sigma^{\frac{1}{m-1}} \l(1-\frac{1}{2m}\r)^{\frac{t-1}{m-1}}}
+
\frac{1}{\sigma(t!)^{2/3}}
.}
Bringing into play the entity
\[
L_{k,q}^*(\sigma):=
L_{k,q}(\sigma)
+\frac{4^{\frac{k}{q}}}{\sigma}
\] 
and noting that 
\[
L_{k+1,q}(\sigma)-4^{\frac{1}{q}}L_{k,q}(\sigma)
=L_{k+1,q}^*(\sigma)-4^{\frac{1}{q}}L^*_{k,q}(\sigma) 
,
\frac{4^{\frac{t-1}{m}}}{\sigma}\leq L_{t-1,m}^*(\sigma)
\]
allows to gain via
~\eqref{eq:fopf} the following inequality,
\[
L_{t,m}^*(\sigma)
\leq
L_{t-1,m}^*(\sigma)
\l(
4^{\frac{1}{m}} 
+
\frac{
(t-1)^{\frac{1}{m-1}}}{
4^{\frac{t}{m(m-1)}}
\l(1-\frac{1}{2m}\r)^{\frac{t-1}{m-1}}}
+
\frac{1}{4^{\frac{t-1}{m}}(t!)^{2/3}}
\r)
.\]
Using the fact $m=f(t)+O(1)$ shows that the middle term in the parenthesis is 
\[\frac{1}{t^{{\log 4}-\frac{1}{2}+o(1)}}.\]
Hence, there exists $c_3>0$, depending at most 
on $K$ and $\psi_K$ such that 
\[
L_{t,m}^*(\sigma)
\leq
\rm{e}^{\frac{c_3}{m}}
L_{t-1,m}^*(\sigma)
,\]
an estimate that
concludes the proof of Proposition~\ref{p:central}.

\bibliographystyle{amsalpha}
\bibliography{deltachar}
\end{document}